\documentclass[12pt,a4paper]{amsart}
\calclayout
\numberwithin{equation}{section}
\allowdisplaybreaks
\theoremstyle{plain}
\newtheorem{thm}{Theorem}[section]
\newtheorem{lem}[thm]{Lemma}
\newtheorem{prop}[thm]{Proposition}
\newtheorem{cor}[thm]{Corollaray}

\newtheorem{ques}[thm]{Question}
\theoremstyle{definition}
\newtheorem{defi}[thm]{Definition}

\newtheorem{rem}[thm]{Remark}

\usepackage{amssymb}
\usepackage{mathtools}
\usepackage{mathrsfs}
\usepackage{hyperref}
\usepackage{xcolor}
\usepackage{comment}
\usepackage{subfigure}

\begin{document}
\title
  [Chemical distance exponent]
  {\large On the chemical distance exponent for the two-sided level-set of the 2D Gaussian free field}
\author{Yifan Gao and Fuxi Zhang}

\address
{Yifan Gao\\
School of Mathematical Sciences\\
Peking University\\
Beijing, China, 100871}
\email{gaoyif@pku.edu.cn}

\address
{Fuxi Zhang\\
	School of Mathematical Sciences\\
	Peking University\\
	Beijing, China, 100871}
\email{zhangfxi@math.pku.edu.cn}



\begin{abstract}
	In this paper we introduce the two-sided level-set for the two-dimensional discrete Gaussian free field. Then we investigate the chemical distance for the two-sided level-set percolation. Our result shows that the chemical distance should have dimension strictly larger than $1$, which in turn stimulates some tempting questions about the two-sided level-set.
\end{abstract}

\maketitle

\section{Introduction}

The discrete Gaussian free field (DGFF) in $\mathbb Z^d, d\ge 3$ is a Gaussian random field with mean zero and covariance given by the Green's function. As a ``strongly'' correlated random field, the
level-set percolation for the DGFF in three dimensions or higher has been extensively studied and
shown to exhibit a non-trivial phase transition by a series of work \cite{MR914444,MR3053773,MR3339867,MR3843421,duminil2020equality}. More precisely, there exists a critical level $0<h_*(d)<\infty$ such that if $h<h_*(d)$, the level-set (a.k.a. excursion set, the random set of points whose value is greater than or equal to $h$) has a unique infinite cluster; if $h>h_*(d)$, the level-set has only finite clusters. 

In this paper, we focus on the two-dimensional DGFF (also called harmonic crystal). However, in two dimensions, we can not define the DGFF in the whole discrete plane since the two-dimensional Green's function blows up, while one can take the scaling limit (the lattice spacing is sent to $0$ while the domain is fixed) to get the continuum Gaussian free field. It is then not possible to investigate the level-set percolation in $\mathbb Z^2$ directly as that in $\mathbb Z^d, d\ge 3$. The right way is to take a big discrete box $V_N$ of side length $N$ and define the two-dimensional DGFF on $V_N$ as a centered Gaussian process with covariance given by the Green's function on $V_N$. Then one can study the connectivity properties of the level-set on $V_N$ as $N$ goes to infinity. However, as shown in \cite{MR3800790}, for any level $h\in\mathbb R$, the level-set above $h$ crosses a macroscopic annulus in $V_N$ with non-vanishing probability as $N$ goes to infinity, which suggests that in some sense there is no non-trivial phase transition for the two-dimensional level-set percolation. Furthermore, the chemical distance (intrinsic distance) between two boundaries of a macroscopic annulus is bounded from above by $N(\log N)^{1/4}$ \cite{MR3800790,MR4112719}. Roughly speaking, the chemical distance has dimension $1$. This inspires us to think that once we truncate the DGFF from two sides to get the so-called two-sided level-set in this paper, whether the chemical distance will be of dimension strictly larger than $1$. Our main result below answers this affirmatively in the sense that if there exists macroscopic (nearest-neighbor) path inside the two-sided level-set, its length must be greater than $N^{1+\varepsilon}$ for some $\varepsilon>0$. Although at present we are not able to show that there could exist some macroscopic path inside the two-sided level-set (This is not obvious, see Question~\ref{que:q1} below), our result gives some expected fractal structure for the two-sided level-set, which is drastically different from the (one-sided) level-set.

Next, we introduce our model and then state our main result. 
For each positive integer $N$, let {$V_{N}=[-N/2,N/2]^2\cap \mathbb Z^2$}. Denote by $\{ \eta^{V_{2N}}(v): v\in V_{2N} \}$ the discrete Gaussian free field (DGFF) on $V_{2N}$ with Dirichlet boundary conditions, which is a mean-zero Gaussian process, vanishing on the boundary, with covariance given by
\[
\mathbb E \eta^{V_{2N}} (u) \eta^{V_{2N}} (v) = G_{2N} (u,v) \quad \text{ for } u,v \in V_{2N},
\]
where $G_{2N}(u,v)$ is the Green's function of the two-dimensional simple random walk in $V_{2N}$. As usual, we restrict ourselves to consider the DGFF $\eta^{V_{2N}}$ on $V_N$ to avoid boundary issues (This can be made more general, see Remark~\ref{rem:boundary-issue}). Suppose $\lambda > 0$. Let
\begin{equation}
{\Lambda_{N,\lambda}} : = \{ v\in V_N: |\eta^{V_{2N}}| \le \lambda  \}.
\end{equation}
We say that a vertex $v\in V_N$ is $\lambda$-open if $\left|\eta^{V_{2N}}(v)\right|\le\lambda$, and interpret $\Lambda_{N, \lambda}$ as the ``two-sided'' level set. Let
 \begin{equation} \label{Eq.defnPke}
{\mathcal{P}_{N}^{\kappa, \epsilon} }= \left\{P: P \text { is a path in } V_{N},\|P\| \geq \kappa N, \text { and }|P| \leq N^{1+\epsilon}\right\},
 \end{equation}
where $\|P\|=\|x-y\|$ if $P$ is a path from $x$ to $y$ and $|P|$ is the length of $P$. We say that $P$ is $\lambda$-open if so is every vertex in $P$. Our main result is the following theorem.

\begin{thm}\label{thm:1.1}
	For each $\lambda>0$, there exists  $\epsilon=\epsilon(\lambda)>0$ such that for every $\kappa\in (0,{1})$,

\begin{equation}\label{eq:complement-event}
\lim _{N \rightarrow \infty}\mathbb{P}\big( P \text { is } \lambda\text{-open for some } P \in {\mathcal{P}_{N}^{\kappa, \epsilon} }\big)=0.
\end{equation}
\end{thm}

\begin{rem}\label{rem:lambda_0}\label{rem:boundary-issue}
	The choice of working with $V_{2N}$ is for convenience; the above theorem holds with { $V_{2N}$ and $V_{N}$ being respectively replaced with $V_{N}$ and $V_{\delta N}$} for any fixed $0<\delta<1$. 
\end{rem}

\begin{rem}
	Theorem~\ref{thm:1.1} also holds for the Gaussian free field on metric graphs since percolation on the metric graph is dominated by the percolation on the integer lattice.
\end{rem}

\begin{rem}
	Note that it suffices to show \eqref{eq:complement-event} holds for large $\lambda$ since the event is increasing in $\lambda$. In fact, we also obtain some quantitive results, namely, it is able to take $\varepsilon(\lambda)=e^{-a\lambda^2}$ for some absolute constant $a>0$ (see \eqref{eq:eps-lambd}), and the probability in \eqref{eq:complement-event} decays faster than $N^{-c\lambda^{-2}}$ for some absolute constant $c>0$ (see \eqref{eq:decay-rate}).
\end{rem}

\begin{rem}
	Furthermore, our method is still effective if $\lambda$ depends on $N$. For example, taking $\lambda=\lambda_N=\sqrt{(2a)^{-1}\log\log N}$, then Theorem~\ref{thm:1.1} shows that for any $\lambda_N$-open path with macroscopic distance, its length should be at least $N(\log N)^{1/2}$. On the other hand, it has been shown in \cite[Theorem 2]{MR1880237} that the maximum of the DGFF is at most $2\sqrt{\frac{2}{\pi}}\log N$ with probability tending to $1$, see \cite{MR4043225} for more about the level-set at heights proportional to the absolute maximum. By symmetry, we see that if we take $\lambda_N=2\sqrt{\frac{2}{\pi}}\log N$, then all points are $\lambda_N$-open with overwhelming probability. Our result stimulates the tempting question of determining the borderline with respect to $\lambda_N$ for linear growth of the chemical distance.
\end{rem}

\begin{ques}\label{que:q1}
	Does there exist a large constant $\lambda>0$ such that with non-vanishing probability there exists a $\lambda$-open path $P$ in $V_N$ with $\|P\| \geq \kappa N$? 
\end{ques}
Let $p(\lambda, N)$ denote the probability of the above event.
Here, we use non-vanishing to mean that $\inf_{N}p(\lambda, N)>0$. It is readily to see that if $\lambda$ is sufficiently small such that $\rho:=\mathbb P(|Z(0,4)|\le\lambda)<\frac14$ where $Z(0,4)$ is a Gaussian random variable with mean $0$ and variance $4$, then $p(\lambda, N)\le N^2(4\rho)^{\kappa N}$ which vanishes exponentially fast in $N$. However, to show it is not the case for large constant $\lambda$ is quite non-trivial. To the best of our knowledge, this question has not been answered yet. We expect that there exists a non-trivial ``phase transition'' for the two-sided level-set percolation. We would like to mention that recently the authors in \cite{ding2020crossing} show that the probability for (one-sided) level-set crossing a rectangle is bounded away from $0$ and $1$.
Another closely related work to this direction is \cite{MR3163210}, in which the ``two-sided'' level-set in $\mathbb Z^d, d\ge 3$ is defined as the random set of points whose absolute value is larger than $h$ (note that this is contrary to our convention for the two-sided level-set), and the associated critical value is proved to be finite for all $d\ge 3$. 

\begin{ques}
	Is there a way to take a scaling limit of the two-sided level-set?
\end{ques}
This might be reminiscent of the Schramm-Sheffield contour line of DGFF in \cite{MR2486487}, which is shown to converge in distribution to $\mathrm{SLE}_4$, as well as the bounded-type thin local sets (BTLS) constructed in \cite{MR3936643} and the first passage set (FPS) introduced in \cite{MR4091511}. Dynkin’s isomorphism enables us to relate the absolute value of the DGFF to the occupation field generated by the random walk loop soup (see \cite{MR2815763,MR3502602}), so all the above questions can be understood in terms of the loop-soup percolation with respect to the occupation field (see \cite{MR2979861,MR3547746,MR3941462} for more related works about the loop-soup). We hope to find some appealing connections between the two-sided level-set and the objects we mentioned.

\subsection{Background}\label{subsec:background}
 The two-dimensional Gaussian free field (GFF) is an important object in statistical physics and the theory of random surfaces \cite{MR2322706}. As {the analog of the Brownian motion with two-dimensional time parameter}, it demonstrates fractal structures in many aspects \cite{MR2486487,MR3947326,MR4019914,MR4076090}. From the perspective of the level-set percolation of the DGFF, we will focus on the chemical distance, which plays a crucial role in {the study of} the fractal structure of clusters in the theory of percolation.

Roughly speaking, the chemical distance is the graph distance on the induced (random) subgraph in some probability models. For instance, the chemical distance for classic percolation models can be defined as the length of the shortest path inside open clusters \cite{MR750568,havlin1985chemical}. However, estimating the chemical distance is quite difficult, which often requires subtle analysis of the structure of the shortest path. Especially, for the two-dimensional critical Bernoulli percolation model, physicists expect that there exists an exponent $d_{\mathrm{min}}$  such that
\begin{equation}
\ell\sim r^{d_{\mathrm{min}}},
\end{equation}
where $r$ and $\ell$ are respectively the Euclidean distance and the chemical distance between two vertices $x$ and $y$, and only the case $r < \infty$ is considered. However, a rigorous way to clarify the equivalence above, i.e., the precise meaning of ``$\sim$", remains to be an open problem \cite{MR2334202,MR3698744}.
It is expected in the physics community that $d_{\mathrm{min}}$ is universal in the sense that it does not depend on the choice of vertices and the type of lattice. In \cite{MR1712629}, Aizenman and Burchard show that $\ell\ge r^{\eta}$ for some $\eta>1$, which implies that $d_{\mathrm{min}}>1$. {Furthermore}, upper bounds on the chemical distance can be obtained by comparing the shortest horizontal crossing with the lowest crossing \cite{MR3698744,damron2017strict}. Specifically, for the critical Bernoulli bond percolation on the edges of a box of side length $n$, the expectation of chemical distance between the left and right sides of the box is $O(n^{2-\delta}\pi_3(n))$ for some $\delta>0$, where $\pi_3(n)$ is the three-arm probability to distance $n$ \cite{damron2017strict}.

Additionally, in the subcritical and supercritical cases of Bernoulli percolation of dimension $d\ge2$, chemical distance is comparable to Euclidean distance \cite{MR762034,MR1404543,MR1068308,MR2319709}. In the critical case in high dimensions, it is shown that macroscopic connecting paths have dimension $2$ \cite{MR2551766,MR2748397,MR3224297}.

We next turn to some correlated percolation models, which have been intensively studied recently \cite{biskup2004scaling,MR2915665,MR3692311,MR3800790}. In the supercritical case for a general class of percolation models on $\mathbb Z^d$ {($d\ge3$)}, with long-range correlations (e.g., the random interlacements, the vacant set of random interlacements, the level sets of the GFF), the chemical distance behaves linearly as in the case of Bernoulli percolation; {see \cite{MR3390739} for details. However, the methods developed in three dimensions and higher are invalid in two dimensions, since the two-dimensional DGFF is log-correlated. Thus it becomes quite complicated when we consider the level-set percolation of the DGFF in two dimensions. Recently, {it is shown in \cite{MR3800790} that for level-set percolation of the two-dimensional  DGFF, the associated chemical distance between two boundaries of a macroscopic annulus is {$O(Ne^{(\log N)^{\alpha}})$} for any $\alpha>1/2$ with positive probability. Later, the order is improved to $O(N(\log N)^{1/4})$ with high probability, on metric graphs, given connectivity  \cite{MR4112719}.}  Note that these two results imply that the undetermined chemical distance exponent for level sets of the two-dimensional DGFF is expected to be $1$ in any phase.

In this paper, we investigate the two-sided level-set cluster of the DGFF. By Theorem~\ref{thm:1.1}, two vertices $x$ and $y$ has chemical distance $\ge N^{1+\epsilon}$, provided that they are connected. Then it will indicate the fractal structure of the two-sided level-set clusters, contrasting to the afore-mentioned case of no fractality of the level-set clusters.

\subsection{Notation conventions}\label{subsec:notations}
For the sake of the reader, we list some notations here.

For $x=(x_1,x_2), y=(y_1,y_2)\in\mathbb R^2$, let
\[
\|x-y\|=\sqrt{|x_1-y_1|^2+|x_2-y_2|^2} \ \text{ and }  \
|x-y|_{\infty}=|x_1-y_1|\wedge|x_2-y_2| 
\]
Let $d(x,B)=\inf_{y\in B}\|x-y\|$ and $d(B_1,B_2)=\inf_{x\in B_1}d(x,B_2)$. Similarly, we define $ d_{\infty}(x,B)=\inf_{y\in B}|x-y|_{\infty}$ and $d_{\infty}(B_1,B_2)=\inf_{x\in B_1}d_{\infty}(x,B_2)$. 

For $x\in\mathbb R^2$ and $\ell>0$, let
\[
B(x,\ell)=\{ y\in \mathbb R^2:\|x-y\|\le \ell \} \ \text{ and }  \
B_{\infty}(x,\ell)=\{ y\in \mathbb R^2: |x-y|_\infty \le \ell \}.
\]
Denote
\[
V_{\ell}(x)=B_{\infty}(x,\ell/2) \cap\mathbb Z^2.
\]

For $a\in\mathbb R$, let $\lfloor a\rfloor$ be the greatest integer that is at most $a$. For $r\ge 1$, let $[r]=\{1,\cdots,\lfloor r\rfloor\}$. Throughout this paper, let $C_1,C_2,\cdots>0$ be universal constants. Let $K=2^k$ be large but fixed in terms of $N$ and to be chosen later, where $k$ is a positive integer. Recall $\kappa\in (0,1)$ as stated in Theorem \ref{thm:1.1}. Let $m\in\mathbb Z_+$ be such that
\[
K^{m+1} \leq \kappa N<K^{m+2}.
\]
Note that $m\rightarrow\infty$ since $N\rightarrow\infty$ and $K, \kappa$ are fixed.

Suppose $B$ is a box in $\mathbb R^2$ and $B\cap\mathbb Z^2\neq\emptyset$, we denote the lower left corner of $B\cap\mathbb Z^2$ by $z_B$.
For each path $P$ in $\mathbb Z^2$, denote by $x_P$ and $y_P$ the starting and ending vertices of $P$, respectively. Denote $\|P\|=\|x_P-y_P\|$.

\subsection{Outline of the proof}

The general proof strategy we employ in this paper is  multi-scale analysis, which is a classic and powerful method in the percolation theory; see for instance \cite{MR1378847,MR1624084,MR3417515,MR3947326}. In order to apply it to prove \eqref{eq:complement-event}, it requires us to combine a contour argument analogous to \cite[Proposition 4]{MR3800790}, which plays an initial role, with the induction analysis analogous to \cite[Lemma 4.4]{MR3947326}. The former is quite similar to \cite{MR3800790}, while the latter is hard in this paper. The main difficulty lies in planning a proper induction strategy and tackling the fluctuation of the harmonic functions in all scales.

Section \ref{sec:2-pre} is devoted to preliminaries, for the sake of the reader. We will list basic results about the DGFF, and show some facts required in later proofs. We will also review the tree structure of a path constructed in \cite{MR3947326}. Roughly speaking, a path $P$ in scale $K^j$, i.e. $\| P \|$ is comparable to $K^j$, is associated with a tree $\mathcal{T}_P$ of depth $j$. Nodes at level $r$ in $\mathcal T_P$ are identified as disjointed sub-paths of $P$ in scale $K^{j-r}$, and the parent/child relation of nodes corresponds to path/sub-path relation. Tame paths are those looking like straight lines, and untamed ones are those looking like curves (see Definition \ref{def:tame}). Then, the fact is that untamed nodes are rare in $\mathcal{T}_P$ for all $P \in \mathcal{P}^{\kappa, \epsilon} $ \cite[Proposition 3.6]{MR3947326}, where
 \[
\mathcal{P}^{\kappa, \epsilon} = \left\{P: P \text { is a path in } V_{N},\|P\| \geq \kappa N, \text { and }|P| \leq N^{1+\epsilon}\right\}
 \]
is defined in \eqref{Eq.defnPke}, and we drop the subscript $N$ for brevity in the context below. Therefore, it remains to show that it is unlikely that tame nodes are all $\lambda$-open, which is actually the essential ingredient of Theorem~\ref{thm:1.1}.

In Section \ref{sec:3-open-path}, we will deal with the contour argument. Concretely, we will show that the probability of there existing a tame and open path started in a fixed box decays stretched-exponentially in $K$ (see Theorem \ref{thm:r=0}). To carry this out, note that the existence of a tame and open path implies that a parallelogram $D$ with aspect ratio $O(K)$ has an open crossing. Next, we cut $D$ into $O(\sqrt K)$ sections uniformly, and extract a  sub-parallelograms $D_i$ with aspect ratio $O(1)$ from the middle of each section (see Figure \ref{fig:LDP}). Then, exponential decay follows from the following two facts. One is that with positive probability, a parallelogram with aspect ratio $O(1)$ has no open crossings (Lemma~\ref{lem:para-crossing}). The other is that the Gaussian values in different $D_i$'s are roughly independent.

In Section \ref{sec:Multi-scale analysis}, we will deal with the induction analysis. Recall that $P$ of scale $K^j$ corresponds to a tree $\mathcal T_P$ of depth $j$. Thus the ratio of tame and open leaves in $\mathcal T_P$ is the average of those in $\mathcal T_{P^{(i)}}$'s, where $P^{(i)}$'s are the children of $P$ and are of number at least $K$. Since $K$ is chosen large, one can apply a large deviation analysis (see Theorem \ref{thm:r=1} and Theorem \ref{thm:xi-bound}). However, we will encounter some technicalities during the proof. Concretely, we need to control the fluctuation of harmonic functions at all scales in an efficient way, so that one can translate the open property into a demand on the GFF at every sub-scale. To this goal, for each level $0\le r\le j$, corresponding to scale $K^{j-r}$, we choose the threshhold $\varepsilon_r$ (see~\eqref{eq:epsilon}) to be large enough to make sure that the harmonic functions at scale $K^{j-r}$ exceed $\varepsilon_r$ with probability decaying sufficiently fast (see Lemmas~\ref{lem:E0}, \ref{lem:E-1} and \ref{lem:E-r}), but on the other hand, $\varepsilon_r$ is not too large in the sense that $\varepsilon_r$'s is a summable sequence. The balance of these two parts ensures that our strategy works.

\section{Preliminaries}\label{sec:2-pre}
There are some basic facts about the two-dimensional discrete Gaussian free field which will be used intensively throughout this paper. For completeness, we will introduce them in Section \ref{subsec:DGFF}. In Section \ref{subsec:tree-structure}, we will collect the results we need from \cite{MR3947326}, including the tree structure of a path (Proposition \ref{prop:tree}) and the upper bound on the total flow through untamed nodes in the associated tree (Lemma \ref{lem:untame-flow}).

\subsection{Properties of two-dimensional DGFF}\label{subsec:DGFF}
In this section, we give a rigorous definition for the DGFF and review some standard estimates about the DGFF. Let $B\subseteq \mathbb Z^2$ be finite and non-empty. Denote by $\{ \eta^B(v): v\in B \}$ the DGFF on $B$ with Dirichlet boundary conditions. It is a mean-zero Gaussian process that vanishes on the boundary $\partial B=\{u\in B:\|u-v\|=1 \text{ for some } v\in B^c\}$, with covariance given by
\[
\mathbb E \eta^B(u)\eta^B(v)=G_B(u,v) \quad \text{ for } u,v\in B,
\]
where $G_B(u,v)$ is the Green's function associated with a simple random walk in $B$, i.e., the expected number of visits to $v$ before reaching $\partial B$ for a discrete simple random walk started at $u$. Without loss of generality, we always assume $\eta^B|_{B^c}=0$. 

To eliminate boundary issues, we will need to consider vertices that have at least an appropriate distance from the boundary. For this purpose, fix $\chi=\frac{1}{10}$, and if $B\subseteq\mathbb Z^2$ is a box of side length $L$, define the box $B^{\chi}:=\{z\in B:d_{\infty}(z,\partial B)>\chi L\}.$

The next lemma says that the DGFF is log-correlated, which can be found in \cite[Eqaution (4)]{MR3947326}.
\begin{lem}\label{lem:log-corr}
Suppose that $B\subseteq\mathbb Z^2$ is a box of side length $L$. There is a universal constant $C_1>0$ such that
\[
\left|\mathbb E\left(\eta^B(u)\eta^B(v)\right)-\frac{2}{\pi}\log\frac{L}{|u-v|_{\infty}\vee 1}\right|\le C_1 \quad \text{for all } u,v \in B^{\chi}.
\]
\end{lem}

The next lemma is the well-known Markov property of the DGFF. A version can be found in \cite[Section 2.2]{MR3947326}. 
\begin{lem}\label{lem:DMP}
	Let $D$ be a finite subset of $\mathbb Z^2$, and $B\subseteq D$. Let $\eta^D$ be the DGFF on $D$, $H^B$ be the conditional expectation of $\eta^D$ given $\eta^D|_{B^c\cup\partial B}$. Then
	\[
	\eta^B:=\eta^D-H^B
	\]
	is a version of the DGFF on B, and it is independent of $H^B$. In other words, $\eta^D=\eta^B\oplus H^B$ is an orthogonal decomposition.
\end{lem}

For the next lemma, we quote a version suited to our needs, which follows straightforwardly from the version in \cite[Lemma 3.10]{MR3433630}.
\begin{lem}\label{lem:H^2}
	In addition to the assumptions in Lemma \ref{lem:DMP}, we further assume that $B$ is a box of side length $L$. Then
	\begin{equation}
	\mathbb E\left( H^B(x)-H^B(y) \right)^2\le C_2\frac{|y-x|_{\infty}}{L}\quad \text{ for all } x,y \in B^{\chi},
	\end{equation}
	where $C_2>0$ is a universal constant.
\end{lem}

Next, we estimate the difference of harmonic functions. It will be intensively used for the rest parts of this paper. 
\begin{lem}\label{lem:fluct-H}
	In addition to the assumptions in Lemma \ref{lem:H^2}, we further assume that $U$ is a box of side length $\ell$ in $B^{\chi}$. There exists a universal constant $C_3>0$ such that if $\varepsilon\ge C_3\sqrt{\frac{\ell}{L}}$, then for all $z\in U$,
	\[
	\mathbb P\Big(\left|H^B(x)-H^B(z)\right|\ge \varepsilon \text{ for some } x\in U \Big)\le
	4\exp\left\{-\frac{\varepsilon^2 L}{8C_2\ell}\right\}.
	\]
\end{lem}

We need the following two lemmas to prove Lemma \ref{lem:fluct-H}.

\begin{lem}[{Dudley’s inequality, \cite[Lemma 4.1]{MR1088478}}]\label{lem:dudley}
	Let $U\subseteq \mathbb Z^2$ be a box of side length $\ell$ and $\{G_w: w\in U\}$ be a mean zero Gaussian field satisfying
	\[
	\mathbb{E}\left(G_{z}-G_{w}\right)^{2} \leq|z-w|_{\infty} / \ell \quad \text { for all } z, w \in U.
	\]
	Then $\mathbb E\max_{w\in U}G_w\le C_4$, where $C_4>0$ is a universal constant.
\end{lem}

\begin{lem}[{Borell–Tsirelson inequality, \cite[Lemma 7.1]{MR3184689}}]\label{lem:borell}
	Let $\{ G_z: z\in X\}$ be a Gaussian field on a finite index set $X$. Set $\sigma^2=\max_{z\in X}\mathrm{Var}(G_z)$. Then
	\[
	\mathbb{P}\left(\left|\max _{z \in X} G_{z}-\mathbb{E} \max _{z \in X} G_{z}\right| \geq a\right) \leq 2 e^{-\frac{a^{2}}{2 \sigma^{2}}} \quad \text { for all } a>0.
	\]
\end{lem}

\begin{proof}[Proof of Lemma \ref{lem:fluct-H}]
	Note that $U\subseteq B^{\chi}\subseteq B\subseteq D$. Fix $z\in U$. For $x\in U$, let $G_x=H^B(x)-H^B(z)$. By Lemma \ref{lem:H^2}, for $x,y\in U$,
	\begin{gather}
	\mathbb{E}\left(G_{x}-G_{y}\right)^{2} \leq C_2\frac{|y-x|_{\infty}}{L}
	\le \frac{C_2\ell}{L}\cdot\frac{|y-x|_{\infty}}{\ell}, \label{eq:Gx-Gy}\\
	\mathrm{Var}(G_x)\le C_2\frac{|x-z|_{\infty}}{L}\le \frac{C_2\ell}{L}. \label{eq:Gx}
	\end{gather}
	By \eqref{eq:Gx-Gy} and  Lemma \ref{lem:dudley}, we have $\mathbb E\max_{x\in U} G_x\le \frac{C_3}{2}\sqrt{\frac{\ell}{L}}$,
	where $C_3=2C_4\sqrt{C_2}$. Combining the symmetry of Gaussian distribution, we have
	\begin{align*}
	&\mathbb P\Big(\left|H^B(x)-H^B(z)\right|\ge \varepsilon \text{ for some } x\in U \Big)\\
	\le&2\mathbb P\Big(G_x\ge \varepsilon \text{ for some } x\in U \Big)
	\le 2\mathbb P\left( \max _{x \in U} G_{x}-\mathbb{E} \max _{x \in U} G_x\ge \varepsilon-\frac{C_3}{2}\sqrt{\frac{\ell}{L}} \right).
	\end{align*}
	Noting that $\varepsilon\ge C_3\sqrt{\frac{\ell}{L}}$ and applying Lemma \ref{lem:borell}, we obtain
	\begin{align*}
	&\mathbb P\Big(\left|H^B(x)-H^B(z)\right|\ge \varepsilon \text{ for some } x\in U \Big)\\
	\le&2\mathbb P\left( \max _{x \in U} G_{x}-\mathbb{E} \max _{x \in U} G_x\ge \frac{\varepsilon}{2}\right) \le 4 \exp\left\{-\frac{\varepsilon^{2}}{8 \sigma^{2}}\right\},
	\end{align*}
	where $\sigma^2=\max_{x\in U}\mathrm{Var}(G_x)\le\frac{C_2\ell}{L}$ by \eqref{eq:Gx}. This concludes the lemma.
\end{proof}

We will need the following standard estimates on simple random walks. We refer the reader to \cite[Lemma 1]{MR3800790} for a similar derivation.
\begin{lem}\label{lem:green's}
	Let $\ell_1>0$, $\ell_2\ge\ell_1+2$ and $z\in \mathbb Z^2$. Suppose $V_{\ell_1}(z)\subseteq D\subseteq V_{\ell_2}(z)$. Then for all $u\in\partial V_{\ell_1}(z)$,
	\[
	\sum_{v\in \partial V_{\ell_1}(z)} G_{D}(u,v)\le\sum_{v\in \partial V_{\ell_1}(z)} G_{V_{\ell_2}(z)}(u,v)\le 2(\ell_2-\ell_1).
	\]
\end{lem}

\subsection{The tree structure associated with a path}\label{subsec:tree-structure}
We now briefly recall some facts from \cite[Section 3]{MR3947326} in this section. For an integer $r\ge 1$, let
 \[
\mathcal{B D}_{r}=\left\{\left[a r-\frac{1}{2},(a+1) r-\frac{1}{2}\right] \times\left[b r-\frac{1}{2},(b+1) r-\frac{1}{2}\right]: a, b \in \mathbb{Z}\right\}.
\]
Note that $(B\cap\mathbb Z^2)$'s partition $\mathbb Z^2$, where $B$ is taken over $\mathcal{B D}_{r}$. Define the sets of paths 
 \[
 {\mathcal{S} \mathcal{L}_{0}:= \mathbb{Z}^{2}}, \ \mathcal{S} \mathcal{L}_{j}:=\left\{P: 1 \leq \frac{1}{K^{j}}\|P\| \leq 1+\frac{1}{K}, P \subseteq B\left(x_{P},\|P\|\right)\right\} \quad \text { for all } j \geq 1,
 \]
recalling $x_P$ and $y_P$ are the two ends of $P$.
If $P\in \mathcal{SL}_j$, $P$ is said to be in scale $K^j$. 

 \begin{defi}\label{def:tame}
For $j\ge0$ and each $P\in \mathcal{SL}_{j+1}$, let

\centerline{$E(P):=\left\{z \in \mathbb{R}^{2}:\left\|x_{P}-z\right\|+\left\|y_{P}-z\right\| \leq\left(1+\frac{2}{K^{2}}\right)\|P\|\right\},$}
\centerline{and $\tilde{E}(P)=\left\{z \in \mathbb{R}^{2}: d(z, E(P)) \leq 4 K^{j}\right\}$.}
\noindent A path $P$ is said to be tame if $P\subseteq \tilde{E}(P)$ and untamed otherwise.
 \end{defi}
Note that only when a path is in scale $K^j$ with $j\ge 1$ could we say it is tame or untamed.

\begin{prop}[{\cite[Proposition 3.1]{MR3947326}}]\label{prop:tree}
Suppose that $j\in[m-2]$ and $P\in\mathcal{SL}_{j+1}$. Then, there exists $\ell\in[K^j,(1+\frac1K)K^j]$, a positive integer $d$, and disjoint child-paths $P^{(i)}$of $P$ for $i\in [d]$ such that the following hold.
\begin{itemize}
	\item[(a)] $d\ge K$.
	\item[(b)] Each box in $\mathcal{BD}_{K^j}$ is visited by at most $12$ sub-paths of the form $P^{(i)}, i \in[d]$.
	\item[(c)] $P^{(i)} \in \mathcal{S} \mathcal{L}_{j}$ for each $i \in[d]$.
\end{itemize}
Furthermore, $d\ge \frac12\|P\|$ for $j=0$; and one can extract $d_0$ disjoint sub-paths in $\mathcal{S} \mathcal{L}_{m-1}$ from $P$ with $\|P\| \geq \kappa N$ such that (b) holds with $j=m-1$, where
$
d_{0}:=\left\lfloor\frac{\kappa N}{K^{m-1}}\right\rfloor \geq K.
$
\end{prop}

Fix $P\in\mathcal{SL}_{j}$. The tree $\mathcal{T}_P$ associated with $P$ is constructed as follows. The nodes of $\mathcal{T}_P$ correspond to a family of sub-paths of $P$, where the parent/child relation in $\mathcal{T}_P$ corresponds to path/sub-path relation in the plane by Proposition \ref{prop:tree}. In particular, the root denoted by $\rho$ corresponds to $P$ and the leaves denoted by $\mathcal{L}$ correspond to vertices on $P$. Denote the level of a node $u$ by $L(u)$ with $L(\rho)=0$ and identify $u$ as a sub-path in $\mathcal{SL}_{j-L(u)}$ with $d_u$ children. Each node as a path enjoys the properties in Proposition \ref{prop:tree}.

Especially, $P$ with $\|P\| \geq \kappa N$ is associated with a tree $\mathcal{T}_P$ of depth $m$. Let $\theta_P$ be the unit uniform flow on $\mathcal{T}_P$ from $\rho$ to $\mathcal{L}$, with $\theta_{P}(\rho)=1$ and $\theta_{P}(v)=\frac{1}{d_{u}} \theta_{P}(u)$ if $v$ is a child of $u$. For $\delta\in(0,1)$, let
\begin{equation}\label{eq:kap-delta-K}
\mathcal{P}^{\kappa, \delta, K}:=\left\{P: P \text { is a path in } V_{N},\|P\| \geq \kappa N \text { and }|P| \leq N^{1+\frac{\delta}{K^{2} k}}\right\}.
\end{equation}

\begin{lem}[{\cite[Proposition 3.6]{MR3947326}}]\label{lem:untame-flow}
	For each $P\in\mathcal{P}^{\kappa, \delta, K}$,
	\[
	\sum_{u: 1 \leq L(u) \leq m-1} \theta_{P}(u) 1_{ \{u \text{ is untamed}\}} \leq 2 \delta m.
	\]
\end{lem}

At the end of this section, we give some definitions that are similar to those in \cite{MR3947326}. Define
\[
\mathcal{P}_{j}\left(B\right) :=\left\{P \in \mathcal{S} \mathcal{L}_{j} : x_{P} \in B \right\},
\]
\[
T_j(B):=\left\{ P\in\mathcal{P}_{j}(B): P\text{ is tame} \right\},
\]
\[
\mathrm{END}_{j} :=\big\{B: B \in \mathcal{B} \mathcal{D}_{K^{j-2}}\text{ and }B\cap V_{N} \neq \emptyset\big\}.
\]
For $B, B' \in \mathrm{END}_j$, define $T_j(B,B'):=\{ P\in T_j(B): y_P\in B' \}$. 

\section{Tame paths are unlikely to be open}\label{sec:3-open-path}
In this section, we will use a contour argument to show that the probability of there existing a tame and open path started in a fixed box decays stretched- exponentially in $K$ (see Theorem \ref{thm:r=0}). To this goal, we start by showing that, with positive probability, a parallelogram with aspect ratio $O(1)$ has no open crossings (Lemma \ref{lem:para-crossing} ) in Section \ref{subsec:good-para}. Then in Section \ref{subsec:proof-r=0}, we give the proof of Theorem \ref{thm:r=0} by using Lemma \ref{lem:para-crossing} and estimates about harmonic functions in Lemma \ref{lem:fluct-H}.

Recall that we call a vertex $x\in V_N$ is $\lambda$-open if $\big|\eta^{V_{2N}}(x)\big|\le \lambda$. Next, we extend this definition a bit. For $V\subseteq V_{2N}$ and $\alpha\in\mathbb R$, we say a vertex $x\in V$ is $(V, \lambda, \alpha)$-open if $|\eta^V(x)+\alpha|\le \lambda$. A path $P\subseteq V$ is said to be $(V, \lambda, \alpha)$-open if so is every vertex in $P$. For brevity, we will occasionally use open  to mean $\lambda$-open or $(V, \lambda, \alpha)$-open according to the context.

\begin{thm}\label{thm:r=0}
	For any $\lambda_0>0$, let $\lambda\ge\lambda_0$. There exists $c=c(\lambda_0)>0$ such that the following holds for all $K\ge K_0(\lambda):=e^{c\lambda^2}$. Suppose that $j\in [m-1]$, $B\in \mathrm{END}_j$, $V_{4K^j}(z_{B})\subseteq V\subseteq V_{2N}$, and $\alpha\in\mathbb R$. Then,
    \begin{equation}
    \mathbb P\big(P \text{ is } (V, \lambda,\alpha)\text{-open for some } P\in T_j(B)\big)
    \le e^{-0.01\sqrt{K}}.
    \end{equation}
\end{thm}

\subsection{Good parallelograms}\label{subsec:good-para}
In this section, we consider a closed parallelogram $D$ with corners $(a,b)$, $(a+l,b+h)$, $(a+l,b+h+w)$ and $(a,b+w)$, where $(a,b)\in\mathbb R^2$, $l\ge w\ge 10$ (here $10$ is a somewhat arbitrary choice), and $l\ge h\ge 0$. Especially, we say $D$ is \emph{good} if $a,l\in\mathbb Z$ and $l=16 w$. We call $l, w$, $\theta=\arctan \frac{h}{l}$, and
\[
v_0=\left(\left\lfloor\frac{a+h+l-7w\sin^2\theta}{2}\right\rfloor , \left\lfloor \frac{b+h-l+7w\sin\theta\cos\theta}{2} \right\rfloor\right)
\]
respectively the length, width, angle and anchor of $D$. Note that $\theta\in [0,\frac{\pi}{4}]$ and $v_0\in\mathbb Z^2$. By crossing of good $D$ we mean a path in $D$ connecting the left and right sides of $D$ (see Figure \ref{fig:para}).

Let $V$ be a finite set in $\mathbb Z^2$ and $D\cap\mathbb Z^2\subseteq V$, then let $\mathcal{A}(D,V,\lambda,\alpha)$ be the event that there exists a $(V,\lambda,\alpha)$-open crossing of $D$. The reasoning of the following lemma is analogous to that of \cite[Proposition 4]{MR3800790}.

 \begin{figure}
	\centering
	\includegraphics[scale=0.4]{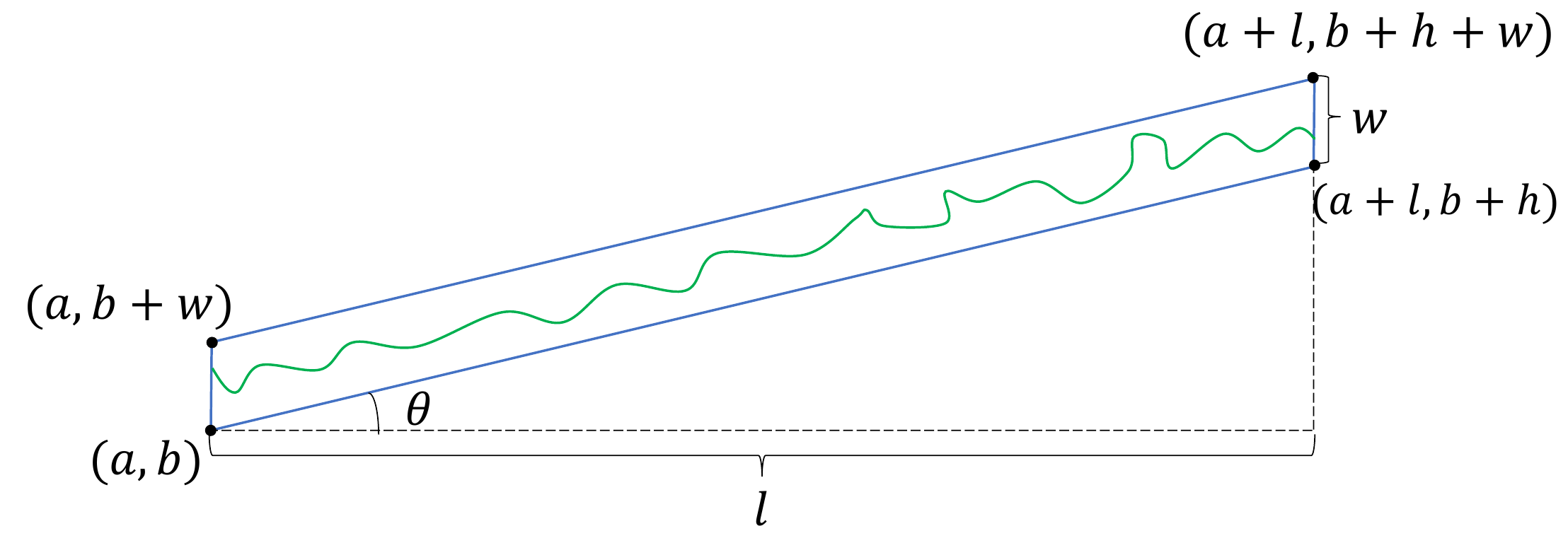}
	\caption{$D$ and its crossing.}
	\label{fig:para}
\end{figure}

\begin{lem}\label{lem:para-crossing}
	For any $\lambda_0>0$, let $\lambda\ge\lambda_0$. There exists $c'=c'(\lambda_0)>0$ such that if 
	\begin{equation}\label{eq:L-w}
	L/w\ge e^{c'\lambda^2},
	\end{equation}
	then for any good parallelogram $D$ with width $w$ and anchor $v_0$, and any  $\alpha\in\mathbb{R}$, we have
	\begin{equation}\label{eq:A}
	\mathbb P\Big( \mathcal{A}\big(D,V_L(v_0),\lambda,\alpha\big) \Big)\le \frac78.
	\end{equation}
\end{lem}


\begin{proof}
	Rotate $D$ around $v_0$ counterclockwise by $i\pi/2$ and denote it by $D_i$, noting $D_0=D$. Since $D$ is good, by our appropriate choice of the anchor $v_0$, $\cup_{i=0}^{3}D_i$ forms an annulus $R$ centered at $v_0$ in $V_{4l}(v_0)$, surrounding $V_{2w}(v_0)$ (see Firgure \ref{fig:rotation}). Let $\mathfrak{C}$ be the collection of all contours in $R$. Here, by contour we mean a path with two endpoints coinciding. We consider a natural partial order on $\mathfrak{C}$: $\mathbf C_1\preceq \mathbf C_2$ if $\mathbf C_1^*\subseteq \mathbf C_2^*$, where $\mathbf C^*$ is the collection of vertices that are surrounded by $\mathbf C$.
	
	\begin{figure}
		\centering
		\includegraphics[scale=0.33]{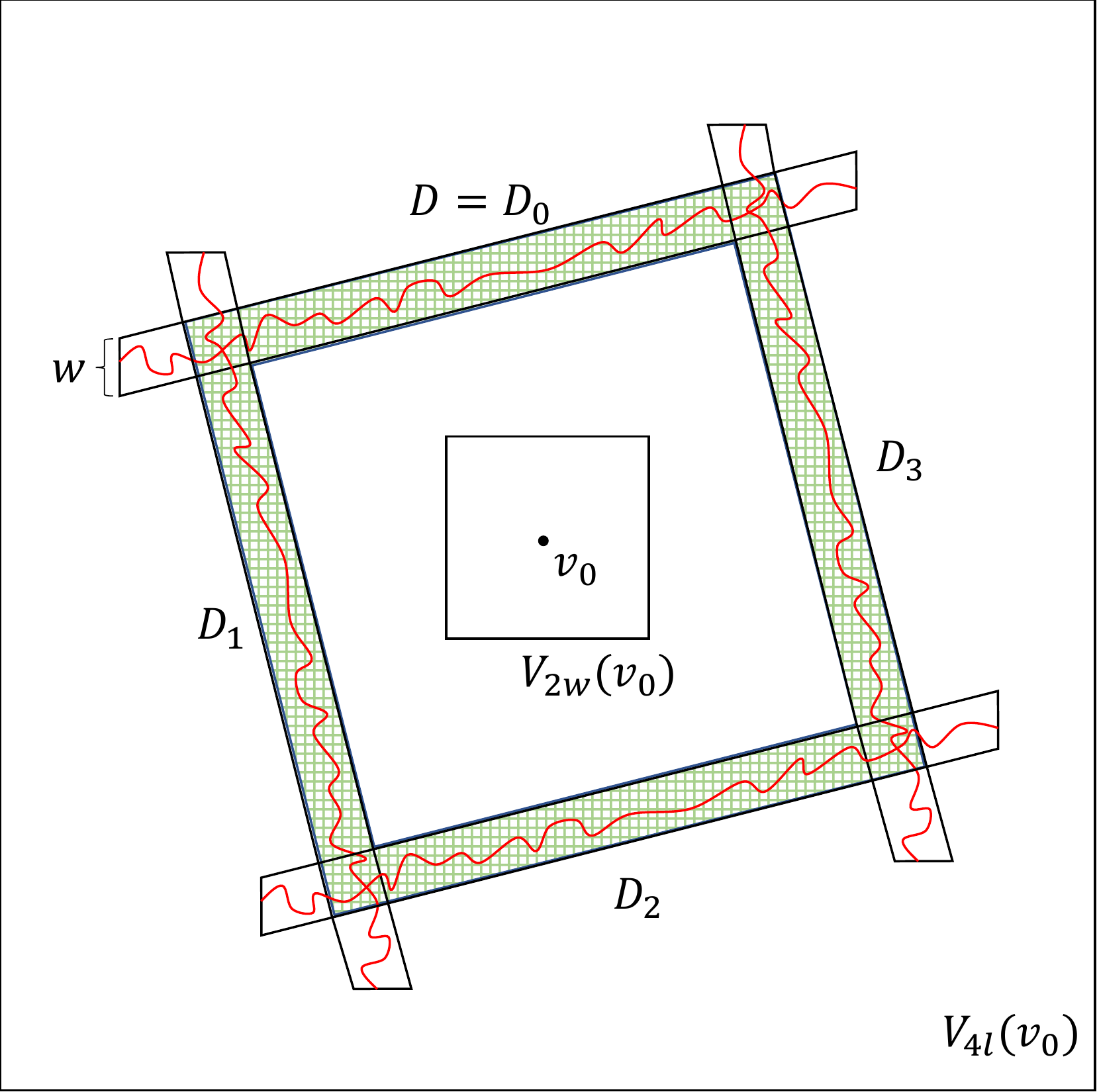}
		\caption{$R$ is the (green) annulus. (Red) curves are open crossings of $D_i$'s.}
		\label{fig:rotation}
	\end{figure}
	
	Denote $V:=V_L(v_0)$ and $\mathcal{A}_i:=\mathcal{A}\big(D_i,V,\lambda,\alpha\big)$, where correspondingly by crossing of $D_i$ for odd $i$, we mean a path in $D_i$ connecting the top an bottom sides of $D_i$. On the event $\cap_{i=0}^3 \mathcal{A}_i$, we can find at least one $(V,\lambda,\alpha)$-open contour in $R$. Let $\mathscr{C}$ be the random subset of $\mathfrak{C}$ consisting of all open contours in $R$. Then, the partial order above generates a well-defined unique maximum contour on $\mathscr{C}$, which is denoted by $\mathcal C$. To the goal, it remains to show
	\begin{equation}\label{eq:C*}
	\mathbb P\big( \mathscr{C}\neq\emptyset\big)\le \frac12.
	\end{equation}
	Assuming \eqref{eq:C*} holds, noting $\cap_{i=0}^3 \mathcal{A}_i\subseteq\{ \mathscr{C}\neq\emptyset \}$ and $\mathbb P\left( \mathcal{A}_i\right)=\mathbb P\left( \mathcal{A}_0 \right) $ by rotation invariance, one has $4\big(1-\mathbb P(\mathcal{A}_0)\big)\ge 1-\mathbb P\big(\cap_{i=0}^3 \mathcal{A}_i\big)\ge1/2$, completing the proof.
	Next, we will prove \eqref{eq:C*}. Denote
	\[
	X:=\frac{1}{|\partial V_{2w}(v_0)|}\sum_{u\in\partial V_{2w}(v_0)}\left(\eta^{V}(u)+\alpha\right).
	\]
    By Lemma \ref{lem:log-corr}, if \eqref{eq:L-w} is satisfied for sufficiently large $c'$, then
	\[
	\mathbb E\eta^{V}(u)\eta^{V}(v)\ge \frac{2}{\pi}\log\left(\frac{L}{2w}\right)-C_1\ge\frac{1}{\pi}\log\left(\frac{L}{2w}\right) \  \text{ for all } u,v\in \partial V_{2w}(v_0).
	\]
	It follows that
	\begin{equation}\label{eq:var-X}
	\mathrm{Var}(X)\ge \frac{1}{\pi}\log\left(\frac{L}{2w}\right).
	\end{equation}
	For a deterministic contour $\mathbf C\in\mathfrak{C}$, let $\hat{\mathbf {C}}=(V\backslash\mathbf C^*)\cup\mathbf C$ be the set of points outside $\mathbf C$ but within $V$. Denote
	$
	\mathcal{F}_{\hat{\mathbf C}}:=\sigma\big\{ \eta^{V}(x): x\in \hat{\mathbf  C}\big\}
	$
	and $Y:=X-\mathbb E\big(X|\mathcal{F}_{\hat{\mathbf C}}\big)$.
    Then,
    \begin{equation}\label{eq:Var-Y}
    \mathrm{Var}(Y)=\frac{1}{|\partial V_{2w}(v_0)|^2}\sum_{u,v\in\partial V_{2w}(v_0)}G_{\mathbf C^*}(u,v)\le 16,
    \end{equation}
    where we have used
    $
    \sum_{v\in\partial V_{2w}(v_0)}G_{\mathbf C^*}(u,v)\le 2(4l-2w)
    $
    by setting $D=\mathbf C^*, \ell_1=2w,\ell_2=4l$ in Lemma \ref{lem:green's}.
    Note that for each $u\in\partial V_{2w}(v_0)$,
   \begin{equation}\label{eq:S-tau}
   \mathbb E\big(\eta^{V}(u)|\mathcal{F}_{\hat{\mathbf C}}\big)=\sum_{v\in \mathbf C}\mathbb P^u(S_{\tau}=v)\cdot\eta^{V}(v),
   \end{equation}
  where $\{S_n\}$ is a simple random walk on $\mathbb Z^2$ started from $u$, and $\tau$ is the first time it hits $\mathbf C$.
  By the definition that $\mathcal{C}$ is the outermost open contour in $\mathscr{C}$, one has $\{ \mathcal{C}=\mathbf C \}\in \mathcal{F}_{\hat{\mathbf C}}$.
  On the event $\{\mathcal{C}=\mathbf C\}$, we have $|\eta^{V}(v)+\alpha| \le \lambda$ for all $v\in \mathbf C$. Combined with \eqref{eq:S-tau}, it gives that
  \[
  \Big|\mathbb E\left(\eta^{V}(u)+\alpha\big|\mathcal{F}_{\hat{\mathbf C}}\right)\Big|
  \le \sum_{v\in \mathbf C}\mathbb P^u(S_{\tau}=v)\cdot \big|\eta^{V}(v)+\alpha\big| \le \lambda \  \text{ for all } u\in \partial V_{2w}(v_0),
  \]
   implying
   $
   \big|\mathbb E\big(X|\mathcal{F}_{\hat{\mathbf C}}\big)\big|
   \le \lambda.
   $
   Consequently, $|Y|\le\lambda$ implies $|X|=\big|\mathbb E\big(X|\mathcal{F}_{\hat{\mathbf C}}\big)+Y\big|\le 2\lambda$.
   Noting that $Y$ and $\mathcal{F}_{\hat{\mathbf C}}$ are independent,
   \begin{equation*}
   \mathbb P\big(|X|\le 2\lambda\big| \mathcal{C}=\mathbf C\big)
   \ge\mathbb P\big( |Y|\le\lambda\big|\mathcal{C}=\mathbf C \big)
   =\mathbb P\big(|Y|\le\lambda\big).
   \end{equation*}
   It follows that
   \begin{equation}\label{eq:cal-C}
   \mathbb P\big(\mathscr C\neq\emptyset\big)
   =\sum_{\mathbf C\in\mathfrak{C}}\mathbb P\big(\mathcal C=\mathbf C\big)
   =\sum_{\mathbf C\in\mathfrak{C}}\frac{\mathbb P(|X|\le 2\lambda, \mathcal{C}=\mathbf C)}{\mathbb P(|X|\le 2\lambda | \mathcal{C}=\mathbf C)}\le \frac{\mathbb P(|X|\le 2\lambda)}{\mathbb P(|Y|\le \lambda)}.
   \end{equation}
   Let $\phi_{\sigma^2}$ be the probability density function of a centered Gaussian random variable with variance $\sigma^2$. Set $\sigma_1^2:=\frac{1}{\pi}\log\left(\frac{L}{2w}\right)$. By \eqref{eq:var-X}, \eqref{eq:Var-Y} and \eqref{eq:cal-C},
   \[
   \mathbb P\big(\mathscr C\neq\emptyset\big)\le  \frac{\phi_{\sigma_1^2}(0)\cdot4\lambda}{\phi_{16}(\lambda_0)\cdot2\lambda_0}.
   \]
  Choose large $c'=c'(\lambda_0)$ such that $\sigma_1$ is large enough to make sure the right hand side above is less than $\frac12$. 
   This completes the proof of the lemma.
\end{proof}

\subsection{Proof of Theorem \ref{thm:r=0}}\label{subsec:proof-r=0}
To prove Theorem \ref{thm:r=0}, it suffices to prove the following proposition.

\begin{prop}\label{prop:BB'}
	Let $K\ge K_0(\lambda)$. For all $B, B' \in \mathrm{END}_j$ such that $T_j(B,B')\neq\emptyset$,
	\begin{equation}
	\mathbb P\Big(P \text{ is } (V,\lambda,\alpha)\text{-open for some } P\in T_j(B,B')\Big)
	\le e^{-0.015\sqrt{K}}.
	\end{equation}
\end{prop}

\begin{proof}[Proof of Theorem \ref{thm:r=0}, assuming Proposition \ref{prop:BB'}]
	 Note that for $B\in\mathrm{END}_j$, one can find at most $K^5$ boxes $B'$'s in $\mathrm{END}_j$ such that $T_j(B,B')\neq\emptyset$. By a union bound, for $K\ge K_0(\lambda)$,
	\begin{equation*}
	\mathbb P\Big(P \text{ is } (V,\lambda,\alpha)\text{-open for some } P\in T_j(B)\Big)
	\le  K^5e^{-0.015\sqrt{K}}\le e^{-0.01\sqrt{K}}.
	\end{equation*}
	This completes the proof.
\end{proof}

We formulate ingredients to prove Proposition \ref{prop:BB'} in the remaining context of this section. In what follows, we will always assume that $B, B' \in \mathrm{END}_j$ and $T_j(B,B')\neq\emptyset$. Let $(x,y)$ and $(x',y')$ be lower-left corners of $B$ and $B'$, respectively. Without loss of generality, suppose that $x'-x\ge y'-y\ge 0$.
Then, it is not hard to show the following geometric facts hold for $K\ge 2^{32}$ (see Figure \ref{fig:LDP}).
\begin{enumerate}
	\item[(G1)]
	One can find a parallelogram $D$ with width $w=20K^{j-1}$ and length $K^j/4$ such that every path in $T_j(B,B')$ contains a crossing of $D$, recalling Definition \ref{def:tame} for the tame path and noting $B, B' \in \mathrm{END}_j$;
	\item[(G2)]
	One can extract good $D_i$'s from $D$ for $i\in[\sqrt{K}/8]$ with width $w=20K^{j-1}$ and length $l=16w$ such that $D_i\subseteq V_{4l}(v_i)\subseteq V_i$ for each $i$, where $v_i$ is the anchor of $D_i$, $L=K^{j-1/2}$, $V_i:=V_L(v_i)$, and $V_i$'s are disjoint.
	\item[(G3)]
     Let $U=\cup_i V_i$. Then $U\subseteq V_{4K^j}(z_{B})\subseteq V$, where $V$ is the set in the statement of Theorem \ref{thm:r=0}.
\end{enumerate}
Set $\mathcal{F}_\partial:=\{ \eta^{V}(z):z\in (V\backslash U)\cup\partial U\}$.
Let $H_{\partial}(z):=\mathbb E\left(\eta^{V}(z)\big|\mathcal{F}_\partial\right)$ and
$
\eta^{V_i}(z):=\eta^{V}(z)-H_{\partial}(z) \text{ for all } z\in V_i.
$
By Markov property (Lemma \ref{lem:DMP}), we know that
$
\eta^{V_i}=\{\eta^{V_i}(z):z\in V_i \}
$
is a DGFF on $V_i$ for each $i\in [\sqrt{K}/8]$, and $\eta^{V_i}$'s are mutually independent by (G2), and they are independent of $H_{\partial}$.
Set
$
\varepsilon_0=100\sqrt{C_2},
$
where $C_2$ is defined in Lemma \ref{lem:H^2}.
Denote
\begin{equation}
\mathcal{E}_0=\Big\{ \big|H_{\partial}(z)-H_{\partial}(v_{i})\big|\ge \varepsilon_0 \text{ for some } i\in [\sqrt{K}/8] \text{ and }z\in D_i \Big\}.
\end{equation}
Set
$
C_5=(2\vee C_4)^{32},
$
where $C_4$ is defined in Lemma \ref{lem:dudley}.

\begin{lem}\label{lem:E0}
	Let $K\ge C_5$. Then, $\mathbb P\big(\mathcal{E}_0\big)\le e^{-0.5\sqrt{K}}$.
\end{lem}

\begin{proof}
	Recall that $w=20K^{j-1}$, $l=16w$, $L=K^{j-1/2}$, and $C_3=2C_4\sqrt{C_2}$. For $K\ge C_5$, we have $C_3\sqrt{\frac{4l}{L}}\le\varepsilon_0$.
	Setting $\ell=4l$ and $\varepsilon=\varepsilon_0$ in Lemma \ref{lem:fluct-H}, we have
	\begin{align*}
	&\mathbb P\Big(\big|H_{\partial}(z)-H_{\partial}(v_{i})\big|\ge \varepsilon_0 \text{ for some } z\in D_i \Big)\\
	\le&\mathbb P\Big(\big|H_{\partial}(z)-H_{\partial}(v_{i})\big|\ge \varepsilon_0 \text{ for some } z\in V_{4l}(v_i)\Big)
	\le 4\exp\left\{-\frac{\varepsilon_0^2 L}{32C_2 l} \right\}\le e^{-0.9\sqrt{K}},
	\end{align*}
	where we have used (G2) in the first inequality, and $K\ge C_5\ge 2^{32}$ in the last inequality. By a union bound,
	$
	\mathbb P\big(\mathcal{E}_0\big)\le \frac18 \sqrt{K}e^{-0.9\sqrt{K}}\le e^{-0.5\sqrt{K}}.
	$
\end{proof}

 \begin{figure}
	\centering
	\includegraphics[scale=0.5]{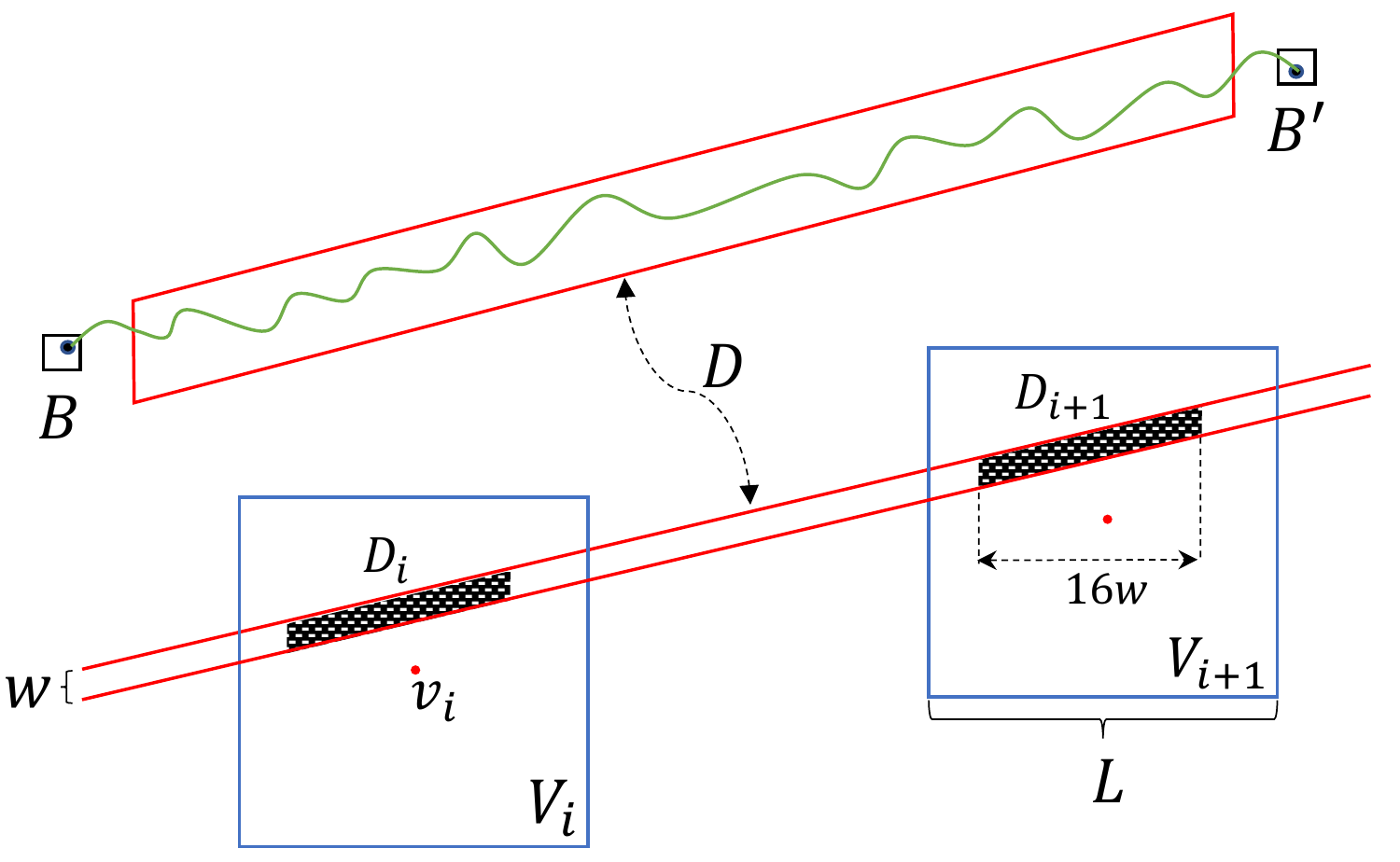}
	\caption{$D$ is the (red) parallelogram, with aspect ratio $O(K)$. 
The parallelograms with (black) shading are $D_i$'s, with aspect ratio $O(1)$. The (blue) squares are $V_i$'s.}
	\label{fig:LDP}
\end{figure}

\begin{proof}[Proof of Proposition \ref{prop:BB'}]
	Let $w=20K^{j-1}$ and $L=K^{j-1/2}$ as above.
	Let $c=c(\lambda_0)$ be a constant such that 
	\begin{equation}\label{eq:K_0}
	K_0(\lambda):=e^{c\lambda^2}\ge 400e^{2c'(\lambda+\varepsilon_0)^2}\vee C_5.
	\end{equation}
	 For $K\ge K_0(\lambda)$, we have
     $
	 L/w\ge e^{c'(\lambda+\varepsilon_0)^2}.
	$
	Then by (G2) and Lemma \ref{lem:para-crossing}, for each $\alpha$,
	\begin{equation}\label{eq:first-term}
			\mathbb P\big( \mathcal A\left(D_i,V_i,\lambda+\varepsilon_0,\alpha\right) \big)\le \frac78 \quad \text{ for all } i,
	\end{equation}
	recalling that $\mathcal A\left(D_i,V_i,\lambda+\varepsilon_0,\alpha\right)$ is the event that there exists a $\left(V_i,\lambda+\varepsilon_0,\alpha\right)$-open crossing of $D_i$ in $V_i$.
	 Note that for $z\in D_i$,
	 \[
	 \eta^{V}(z)=\eta^{V_i}(z)+H_{\partial}(z)
	 =\big(\eta^{V_i}(z)+H_{\partial}(v_i)\big)+\big(H_{\partial}(z)-H_{\partial}(v_i)\big).
	 \]
	 By the triangle inequality, if $|H_{\partial}(z)-H_{\partial}(v_i)|\le\varepsilon_0$ for all $z\in D_i$, then $\left|\eta^{V}(z)+\alpha\right|\le\lambda$
	 implies
	 $
	 \big|\eta^{V_i}(z)+\alpha+H_{\partial}(v_i)\big|\le\lambda+\varepsilon_0.
	 $
	 Thus, on the event $\mathcal{E}_0^c$,
	 \[
	 \mathcal A\subseteq \bigcap_{i\in [\sqrt{K}/8]}\mathcal A\big(D_i,V,\lambda,\alpha\big)
	 \subseteq \bigcap_{i\in [\sqrt{K}/8]}\mathcal A_i,
	 \]
	 where $\mathcal A:=\mathcal A\big(D,V,\lambda,\alpha\big)$, $\mathcal A_i:=\mathcal A\big(D_i,V_i,\lambda+\varepsilon_0,\alpha+H_{\partial}(v_{i})\big)$, and $H_{\partial}(v_{i})$ is regarded as a constant with respect to the DGFF $\eta^{V_i}$ for independence. Therefore,
	 \begin{equation}\label{eq:mathcal-A}
	 \mathbb P\big( \mathcal A \big)\le \mathbb P\Big( \bigcap_{i}\mathcal A_i \Big)+\mathbb P\big( \mathcal{E}_0 \big),
	 \end{equation}
	 where the intersection is over $i\in [\sqrt{K}/8]$. By \eqref{eq:first-term},
     \begin{equation}\label{eq:intersection}
     \mathbb P\Big( \bigcap_{i}\mathcal A_i \Big)
     \le \mathbb E\Big(\mathbb P\Big(\bigcap_i\mathcal A_i\Big| \mathcal{F}_\partial\Big)\Big)
     \le \mathbb E \prod_i \mathbb P\big(\mathcal A_i\big| \mathcal{F}_\partial\big)
     \le \left(\frac78\right)^{\lfloor\sqrt{K}/8\rfloor},
     \end{equation}
     where we have used the conditional independence of $\mathcal A_i$ given $\mathcal{F}_\partial$.
	 Combining (G1), \eqref{eq:mathcal-A}, \eqref{eq:intersection} and Lemma \ref{lem:E0}, for $K\ge K_0(\lambda)\ge C_5\ge 2^{32}$,
	 \[
	 \mathbb P\Big(P \text{ is } (V,\lambda,\alpha)\text{-open for some } P\in T_j(B,B')\Big)
	 \le \left(\frac78\right)^{\lfloor\sqrt{K}/8\rfloor}+e^{-0.5\sqrt{K}} \le e^{-0.015\sqrt{K}}.
	 \]
	 This completes the proof.
\end{proof}

\section{Multi-scale analysis on the hierarchical structure of the path}\label{sec:Multi-scale analysis}
In this section, we will prove Theorem \ref{thm:1.1}. It suffices to prove the theorem for $\lambda\ge\lambda_0$ with $\lambda_0>0$ fixed, see Remark~\ref{rem:lambda_0}.
Note that if there is a $\lambda$-open path $P$ in $\mathcal{P}^{\kappa,\delta,K}$, then all nodes in $\mathcal{T}_P$ are $\lambda$-open. We will prove that this event has probability tending to $0$ as the depth of the tree $m$ tends to infinity, by showing that tame and open nodes are rare (Theorem \ref{thm:xi-bound} below). Note that untamed nodes are rare by Lemma \ref{lem:untame-flow}.

Let $j\in[m-1], B\in \mathrm{END}_j$.
For each $P\in\mathcal P_j(B)$, recall that $\mathcal T_P$ is a tree of depth $j$, associated with $P$. Each node $u$ of $\mathcal T_P$ is identified with a sub-path of $P$, which is also denoted by $u$ to lighten notation. Let $\mathcal T_{P,r}$ be the collection of nodes of level $r$. Note that the root has level $0$. For each $u\in \mathcal T_{P,r}$, there is a unique starting box $B_u\in \mathrm{END}_{j-r}$ containing the starting point of $u$.
Let $\mathscr A$ be the collection of real functions defined on $\cup_{j\in[m-1]}\mathrm{END}_{j}$, i.e., on all end-boxes. We always assume that $\bar\alpha$ is a real function in $\mathscr A$. Note that for any $P$, $\bar\alpha$ induces a function on $\mathcal T_P$ by setting $\bar\alpha_u:=\bar\alpha\big(B_u\big)$ for each $u\in \mathcal T_{P}$.
Let $\theta_P$ be the unit uniform flow on $\mathcal T_P$ from $\rho$ to $\mathcal{L}$ (the definition is just before \eqref{eq:kap-delta-K}), where $\rho$ is the root and $\mathcal{L}$ is the set of leaves. For $\lambda>0, V_{4K^j}(z_{B})\subseteq V\subseteq V_{2N}$, define
\begin{gather}
Y_{P,r,\lambda,\bar\alpha}:=\sum_{u\in \mathcal T_{P,r}}\theta_P(u)1_{\{u \text{ is tame and } (V,\lambda,\bar\alpha_u)\text{-open}\}}, \label{eq:Y-Pr}\\
\xi_{r,\lambda,\bar\alpha,j,B}:=\max\big\{Y_{P,r,\lambda,\bar\alpha}: P\in \mathcal P_j(B)\big\}.\label{eq:xi-Pr}
\end{gather}


Recall $\lambda\ge\lambda_0$ and $K_0(\lambda)=e^{c\lambda^2}$ as a function of $\lambda$ for some $c=c(\lambda_0)>0$.
Noting that
 \[
 \xi_{0,\lambda,\bar\alpha,j,B}=1_{\left\{P \text{ is } (V,\lambda,\bar\alpha_{\rho})\text{-open for some } P\in \mathcal T_j(B) \right\}},
 \]
the next corollary restates Theorem \ref{thm:r=0}.

\begin{cor}\label{cor:r=0}
	Suppose $K\ge K_0(\lambda)$, $j\in[m-1]$, $B\in \mathrm{END}_j$, $V_{4K^j}(z_{B})\subseteq V\subseteq V_{2N}$, and $\bar\alpha\in\mathscr A$. Then,
	\begin{equation}
	\mathbb P(\xi_{0,\lambda,\bar\alpha,j,B}>0)\le e^{-0.01\sqrt{K}}.
	\end{equation}
\end{cor}

As for $r=1$, we have a similar result. Set
$
\varepsilon_{1}=8\sqrt{C_2}$, $ K_1(\lambda)=K_0(\lambda+\varepsilon_{1}).
$
We will prove the following theorem in Section \ref{subsec:4.1}.

\begin{thm}\label{thm:r=1}
	Suppose $K\ge K_1(\lambda)$, $j\in[2,m-1]\cap\mathbb Z$, $B\in \mathrm{END}_j$, $V_{4K^j}(z_{B})\subseteq V\subseteq V_{2N}$, and $\bar\alpha\in\mathscr A$. Then,
	\[
	\mathbb P(\xi_{1,\lambda,\bar\alpha,j,B}>\delta)
	\le e^{-K^{1/8}} \quad  \text{ for all } \delta\ge\delta_1:=\frac12.
	\]
\end{thm}

Before generalizing Theorem \ref{thm:r=1}, let us set our conventions for constants.
Set
\begin{equation}\label{eq:beta-c_r}
\beta=2^{-9}\ \text{ and }\  c_r=(\beta K)^r.
\end{equation}
Define $\{\delta_r : r\ge0\}$ to be
\begin{gather}
\delta_0=0, \quad \delta_1=\frac12; \quad \delta_{r+1}=\delta_r+\Delta_r \ \text{ for all } r\ge 1,\label{eq:delta}\\
\text{ where } \Delta_1=\frac{9\log K}{\beta K^{1/8}}; \quad \Delta_{r+1}=\frac{\log(1+2c_{r})+9\beta^{-1}\log K}{c_{r}}\ \text{ for all } \ r\ge 1.\label{eq:Delta}
\end{gather}
Set
\begin{gather}
\varepsilon_0=100\sqrt{C_2}, \quad \varepsilon_1=8\sqrt{C_2}; \quad \varepsilon_{r+1}=4\sqrt{C_2}\beta^{r/2}\ \text{ for all } \ r\ge 1, \label{eq:epsilon}\\
K_{r+1}(\lambda)=K_r(\lambda+\varepsilon_{r+1})=K_0\left(\lambda+\sum_{i=1}^{r+1}\varepsilon_i\right)
\ \text{ for all } r\ge 0.\label{eq:Kr}
\end{gather}

We will prove the following theorem by induction on admissible pair $(r,j)$ in Section \ref{subsec:4.2}.
\begin{thm}\label{thm:xi-bound}
	The following holds for any pair $(r,j)$ satisfying $r\in[2,m-2]\cap\mathbb Z$ and $j\in[r+1,m-1]\cap\mathbb Z$. For all $K\ge K_r(\lambda)$, $B\in \mathrm{END}_j$, $V_{4K^j}(z_{B})\subseteq V\subseteq V_{2N}$, and $\bar\alpha\in\mathscr A$, we have
	\[
	\mathbb{P}(\xi_{r,\lambda,\bar\alpha,j,B}>\delta)
	\le 2e^{-c_{r-1}(\delta-\delta_r)}
	\quad \text{for all } \delta\ge\delta_r.
	\]
\end{thm}

In other words, by choosing $\delta_r$ as the threshold for total flows through level $r$, the overflows above $\delta_r$ will have an exponential decay uniformly in other parameters.
Furthermore, as in Lemma \ref{lem:E0}, we will use $\{\varepsilon_r : r\ge 1\}$ to bound the fluctuation of harmonic functions at different levels in Lemma \ref{lem:E-1} ($r=1$) and Lemma \ref{lem:E-r} ($r\ge2$), respectively.

\subsection{Proof of Theorem \ref{thm:r=1}}\label{subsec:4.1}
We assume $j\in[2,m-1]\cap\mathbb Z$, $B\in \mathrm{END}_j$, $V_{4K^j}(z_{B})\subseteq V\subseteq V_{2N}$, and $\bar\alpha\in\mathscr A$ in this section.
Define
$
\mathcal{P}_{j, d}(B):=\big\{P \in \mathcal{P}_{j}\left(B\right) : d_{P}=d\big\}.
$
Denote the child-paths of $P$ by $\{P^{(i)}\}_{i\in[d]}$ if $P\in\mathcal{P}_{j, d}(B)$. Note that $d\ge K$ always holds by (a) of Proposition \ref{prop:tree}. Define
\[
\mathrm{END}_{j-1, d} :=\left\{\left\{B_i\right\}_{i \in[d]} \subseteq \mathrm{END}_{j-1} :\left|\left\{i : B_{i} \subseteq \tilde B\right\}\right| \leq 12 \text { for each } \tilde B \in \mathcal{B} \mathcal{D}_{K^{j-1}}\right\},
\]
and for each sequence $\mathcal{S} :=\left\{ B_i \right\}_{i \in[d]}\in\mathrm{END}_{j-1, d}$,
define
\[
\mathcal{P}_{j, \mathcal{S}}(B) :=\big\{P \in \mathcal{P}_{j, d}(B) : P^{(i)} \in \mathcal{P}_{j-1}\left(B_i\right) \text { for all } i \in[d]\big\}.
\]
Furthermore, define
\begin{equation*}
\mathrm{END}_{j-1, d}(B):=\big\{\mathcal S\in \mathrm{END}_{j-1,d}: \mathcal{P}_{j, \mathcal{S}}(B)\neq\emptyset \big\}.
\end{equation*}
For the remainder of this paper, we always assume that $\mathcal{S} :=\left\{ B_i \right\}_{i \in[d]}\in\mathrm{END}_{j-1, d}(B)$ and $d\ge K$. Denote for brevity
\begin{equation}\label{eq:z-V}
z_i:=z_{B_{i}} \ \text{ and } \ V_i:=V_{K^{j-7/8}}(z_i) \ \text { for all } i \in[d].
\end{equation}
Note that $V_i\subseteq V_{4K^j}(z_{B})\subseteq V$ for all $i\in[d]$.
Let $H_i$ be the conditional expectation of $\eta^V$ given $\eta^V\big|_{ V_i^c\cup\partial V_i }$. By Lemma \ref{lem:DMP}, $\eta^{V_i}:=\eta^V-H_i$ is a DGFF on $V_i$ for each $i\in[d]$. Recall $\varepsilon_{1}=8\sqrt{C_2}$ and define
\begin{equation}\label{eq:E_S}
\mathcal{E}_{\mathcal S}=\bigcup_{i\in[d]}\Big\{ \big|H_i(x)-H_i(z_i)\big|\ge\varepsilon_{1} \text{ for some } x\in V_{4K^{j-1}}(z_i)\Big\}.
\end{equation}

For all $P\in\mathcal{P}_{j, \mathcal{S}}(B)$, noting that $P^{(i)}\subseteq V_{4K^{j-1}}(z_i)$ for all $i\in [d]$, and $\eta^{V}(x)=\left(\eta^{V_i}(x)+H_i(z_i)\right)+\big( H_i(x)-H_i(z_i) \big)$ for all $x\in P^{(i)}$; then on the event $\mathcal{E}_{\mathcal S}^c$, by the triangle inequality, $P^{(i)} $ is $(V,\lambda,\alpha_i)$-open implies that it is $\big(V_i,\lambda+\varepsilon_{1},\alpha_i+H_i(z_i)\big)$-open, where $\alpha_i=\bar\alpha(B_i)$ and $H_i(z_i)$ is regarded as a deterministic number with respect to the field $\eta^{V_i}$. Therefore,
\[
Y_{P, 1,\lambda, \bar\alpha}\le \frac1d \sum_{i=1}^{d} Y'_{i,0,\lambda+\varepsilon_{1}}
\] 
for all  $P\in\mathcal{P}_{j, \mathcal{S}}(B)$ on the event $\mathcal{E}_{\mathcal S}^c$,
where
$Y'_{i,0,\lambda+\varepsilon_{1}}$ is the indicator function of the event that 
$P^{(i)}$ is tame and  $\left(V_i,\lambda+\varepsilon_{1},\alpha_i+H_i(z_i)\right)$-open.
This implies that
\begin{equation*}
\zeta_{1,\mathcal{S}}\le
\frac1d \sum_{i=1}^{d}\xi'_{i} \ \text{ on the event } \mathcal{E}_{\mathcal S}^c,
\end{equation*}
where
$
\zeta_{1,\mathcal{S}} :=\max \big\{Y_{P, 1,\lambda, \bar\alpha}: P \in \mathcal{P}_{j, \mathcal{S}}(B)\big\},
$
and
\begin{equation}\label{eq:xi'0}
\xi'_{i}:=1_{ \big\{\text{ there exists a } (V_i,\lambda+\varepsilon_{1},\alpha_i+H_i(z_i))\text{-open path in }  T_{j-1}\left(B_i\right) \big\} }.
\end{equation}
It follows that for all $\delta>0$ and $\mathcal S\in\mathrm{END}_{j-1, d}(B)$,
\begin{equation}\label{eq:1S-delta}
	\mathbb P\big(\{\zeta_{1,\mathcal{S}} >\delta\}\cap\mathcal{E}_{\mathcal S}^c\big)\le \mathbb P\left(\frac1d \sum_{i=1}^{d}\xi'_i>\delta\right).
\end{equation}

Based on an argument similar to \cite[Lemma 4.4]{MR3947326}, we obtain the following lemma.

\begin{lem}\label{lem:average-level0}
	Let $K\ge K_1(\lambda)$. For each $\mathcal{S} :=\left\{ B_i \right\}_{i \in[d]}\in\mathrm{END}_{j-1, d}(B)$, we have
    \[
    \mathbb P\left(\frac1d \sum_{i=1}^{d}\xi'_i>\delta\right)\le e^{ -10^{-4}K^{1/4}\delta d } \quad \text{ for all } \delta\ge\delta_{1}=\frac12.
    \]
\end{lem}

\begin{proof}
Let $\beta_K=(48K^{1/4})^{-1}$.
We will classify $B_i$'s into $\beta_K^{-1}$ groups in the following procedure, such that $V_i$'s in each group are disjoint. Note that if $d_{\infty}(B_{i},B_{i'})\ge 1.5K^{j-7/8}$, then $d_{\infty}(V_i,V_{i'})\ge K^{j-1}$. First, we classify $\mathcal{BD}_{K^{j-1}}$ into $4K^{1/4}=\left( 2K^{j-7/8}/K^{j-1} \right)^2$ families $\tilde{\mathcal G}_s, s\in [4K^{1/4}]$, where $\tilde{\mathcal G}_1$ consists of boxes respectively containing $(2aK^{j-7/8},2bK^{j-7/8})$, $a,b\in\mathbb Z$ and other $\tilde{\mathcal G}_s$'s are its shifts. Let 
\[
\mathcal{G}_{s}:=\big\{B_{i}: i \in[d],\text{ and } B_{i} \subseteq \tilde B \text{ for some } \tilde B \in \tilde{\mathcal G}_s\big\}.
\] 
Then, by (b) of Proposition \ref{prop:tree}, we can classify each $\mathcal{G}_{s}$ into $12$ groups $\mathcal{G}_{s,t}, t\in [12]$, such that for each $s,t$, a box in $\tilde{\mathcal G_s}$ contains at most one $B_{i}$ in $\mathcal G_{s,t}$. Thus, $V_i$'s in each group $\mathcal G_{s,t}$ are disjoint.

 Let $V_{s,t}=\cup_i V_i$ be the union of $V_i$'s with $i$ such that $B_{i}\in \mathcal{G}_{s,t}$. Define the $\sigma$-field generated by the information outside $V_{s,t}$ by
\[
\mathcal F_{s,t}:=\big\{ \eta^{V} (x): x\in (V   \backslash V_{s,t})\cup\partial V_{s,t} \big\}.
\]
Then, conditioned on $\mathcal F_{s,t}$, $\xi'_i$'s in each group $\mathcal{G}_{s,t}$ are mutually independent. Denote \[W_{s,t}:=\prod_{B_{i}\in \mathcal{G}_{s,t}}e^{a\beta_K(\xi'_i-\delta)}, 
\]
where $\delta\ge\delta_1=\frac12$ and $a$ is a positive number to be set.
Then, we have
\begin{equation}\label{eq:W-st}
\mathbb{E}W_{s,t}^{1/\beta_K}
=\mathbb{E}\prod_{B_{i}\in \mathcal{G}_{s,t}}e^{a(\xi'_i-\delta)}
=\mathbb{E}\prod_{B_{i}\in \mathcal{G}_{s,t}}\mathbb{E}\left(e^{a(\xi'_i-\delta)} \big| \mathcal F_{s,t}\right).
\end{equation}

Next, we will estimate $\mathbb{E}\left(e^{a(\xi'_i-\delta)} | \mathcal F_{s,t}\right)$. Since $K\ge K_1(\lambda)= K_0(\lambda+\varepsilon_{1})$, by Corollary \ref{cor:r=0}, $\xi'_i$ is a Bernoulli random variable with
$
\mathbb{P}(\xi'_i=1| \mathcal F_{s,t})\le e^{-0.01\sqrt{K}}=:g(K).
$
Consequently,
\[
\mathbb{E}\left(e^{a(\xi'_i-\delta)} \big| \mathcal F_{s,t}\right)\le e^{a (1-\delta)}g(K)+e^{-a\delta}.
\]
Set $a=\log\left(\frac{\delta}{1-\delta}g(K)^{-1} \right)$ to optimize the above bound, noting
$a\ge\log\left( \frac{\delta_1}{1-\delta_1}g(K)^{-1} \right)=0.01\sqrt{K}>0$.
It follows that
\begin{equation}\label{eq:mathod-1}
\mathbb{E}\left(e^{a(\xi'_i-\delta)} \big| \mathcal F_{s,t}\right)\le  f(\delta)g(K)^{\delta}\le 2g(K)^{\delta},
\end{equation}
where $f(\delta):=\left( \frac{\delta}{1-\delta} \right)^{1-\delta}+\left( \frac{\delta}{1-\delta} \right)^{-\delta}\le 2$. Combined with \eqref{eq:W-st}, this yields
\begin{equation}\label{eq:W-st-2}
\mathbb{E}W_{s,t}^{1/\beta_K}\le \prod_{B_{i}\in \mathcal{G}_{s,t}} \left(2g(K)^{\delta} \right).
\end{equation}
By the Cauchy-Schwarz inequality,
\begin{equation}\label{eq:C-S}
\mathbb E e^{a\beta_K\sum_i(\xi'_i-\delta)}
=\mathbb{E}\prod_{s=1}^{4K^{1/4}}\prod_{t=1}^{12}W_{s,t}
\le \prod_{s=1}^{4K^{1/4}}\prod_{t=1}^{12}\left(\mathbb{E}W_{s,t}^{1/\beta_K}\right)^{\beta_K}.
\end{equation}
Combining \eqref{eq:W-st-2} and \eqref{eq:C-S}, we obtain
\begin{align}
\mathbb P\left(\frac1d \sum_{i=1}^{d}\xi'_i>\delta\right)
&\le \mathbb E e^{a\beta_K\sum_i(\xi'_i-\delta)}
\le\prod_{s=1}^{4K^{1/4}}\prod_{t=1}^{12}\prod_{B_i\in \mathcal{G}_{s,t}}\left(2g(K)^{\delta}\right)^{\beta_K}\nonumber\\
&\le\left(2g(K)^{\delta}\right)^{\beta_Kd}
\le e^{ -10^{-4}K^{1/4}\delta d },\label{eq:method-2}
\end{align}
where the last inequality follows from $d\ge K\ge 2^{32}$ and $\delta\ge \frac12$.
\end{proof}

Recall \eqref{eq:E_S} for the definition of $\mathcal{E}_{\mathcal S}$ and define the event
\begin{equation}\label{eq:E1}
\mathcal{E}_1:= \bigcup_{d\ge K}\bigcup_{\mathcal S\in\mathrm{END}_{j-1, d}(B)}\mathcal{E}_{\mathcal S}.
\end{equation}
To prove Theorem \ref{thm:r=1}, we need to estimate $\mathbb P\big(\mathcal{E}_1 \big)$ in addition. The argument is quite similar to Lemma \ref{lem:E0}.

\begin{lem}\label{lem:E-1}
	Let $K\ge C_5$. Then
	$
	\mathbb P\big(\mathcal{E}_1 \big)
	\le e^{-1.5K^{1/8} }.
	$
\end{lem}
\begin{proof}
	There are at most $K^7$ boxes in $\mathcal{BD}_{K^{j-3}}$ intersecting with some paths in $\mathcal{P}_j(B)$. Denote them by $B_t$'s. For each $B_t$, denote for brevity
	\[
	z_t=z_{B_t}\ \text{ and } \  V_t=V_{K^{j-7/8}}(z_t).
	\]
	Let $H_t$ be the conditional expectation of $\eta^V$ given $\eta^V\big|_{ V_t^c\cup\partial V_t }$. Setting $\ell=4K^{j-1}, L=K^{j-7/8}$ in Lemma \ref{lem:fluct-H}, recalling $C_5=(2\vee C_4)^{32}$, for $K\ge C_5$, we have
	$
	C_3\sqrt{\frac{\ell}{L}}=C_3\sqrt{\frac{4}{K^{1/8}}}\le\varepsilon_1,
	$
	and for all $t$,
	\[
	\mathbb P\Big( \big|H_t(x)-H_t(z_t)\big|\ge\varepsilon_{1} \text{ for some } x\in V_{4K^{j-1}}(z_t) \Big)\le 4e^{ -2K^{1/8}}.
	\]
	Note that $\mathcal{E}_1 $ implies that the fluctuation of $H_t$ in $V_{4K^{j-1}}(z_t)$ is greater than $\varepsilon_{1}$ for some $t$. Thus, we obtain
	$
	\mathbb P\left(\mathcal{E}_1 \right)
	\le  4K^7e^{ -2K^{1/8}}
	\le e^{ -1.5K^{1/8}},
	$
	completing the proof.
\end{proof}

\begin{proof}[Proof of Theorem \ref{thm:r=1}]
It can be seen from the definition that
	\begin{equation}\label{eq:xi1-delta}
	\mathbb P\big(\xi_{1}>\delta\big)
	\le \sum_{d=K}^\infty\sum_{\mathcal S\in\mathrm{END}_{j-1,d}(B)} \mathbb P\big(\{\zeta_{1,\mathcal S}>\delta \}\cap\mathcal{E}_{\mathcal S}^c\big)+\mathbb P\big(\mathcal{E}_1\big),
	\end{equation}
	Note that there are at most $K^7$ boxes in $\mathcal{BD}_{K^{j-3}}$ intersecting with some path in $\mathcal{P}_j(B)$. Therefore, there are at most $K^{7d}$ sequences in $\mathrm{END}_{j-1,d}(B)$.
	Combined with \eqref{eq:1S-delta}, Lemma \ref{lem:average-level0} and Lemma \ref{lem:E-1}, and using a union bound, this yields
	\[
	\mathbb P(\xi_{1}>\delta)
	\le\sum_{d=K}^{\infty}K^{7d}e^{ -10^{-4}K^{1/4}\delta d }+e^{ -1.5K^{1/8}}
	\le e^{ -K^{1/8}},
	\]
	where in the last inequality we have used $\sum_{d=K}^{\infty}K^{7d}e^{ -10^{-4}K^{1/4}\delta d }
	\le e^{-K}$ for $\delta\ge\frac12$ and $K\ge 2^{32}$.
	This completes the proof of the theorem.
\end{proof}

\subsection{Proof of Theorem \ref{thm:xi-bound}}\label{subsec:4.2}
Assume $r\in[2,m-2]\cap\mathbb Z$, $j\in[r+1,m-1]\cap\mathbb Z$, $B\in \mathrm{END}_j,V_{4K^j}(z_{B})\subseteq V\subseteq V_{2N}$, and $\bar\alpha\in\mathscr A$ in this section.
The reasoning of the proof of Theorem \ref{thm:xi-bound} is similar to that of Theorem \ref{thm:r=1}.
Recall $\mathcal{S} :=\left\{B_{i}\right\}_{i \in[d]}\in\mathrm{END}_{j-1, d}(B)$. Compared with \eqref{eq:z-V}, here we set
\[
z_i:=z_{B_{i}} \ \text{ and } \ V_i:=V_{4K^{j-1}}(z_i) \ \text { for all } i \in[d].
\]
 Noting that $V_i\subseteq V$ for all $i$, let $H_i$ be the conditional expectation of $\eta^V$ given $\eta^V\big|_{ V_i^c\cup\partial V_i }$. By Lemma \ref{lem:DMP}, $\eta^{V_i}:=\eta^V-H_i$ is a GFF on $V_i$ for all $i$. Recall $\beta=2^{-9}$ set in \eqref{eq:beta-c_r} and  $\varepsilon_{r+1}=4\sqrt{C_3}\beta^{r/2}$ set in \eqref{eq:epsilon}. Analogous to \eqref{eq:E_S} and \eqref{eq:E1}, we define the events
 \begin{equation}\label{eq:Ei}
  \mathcal{E}_{r+1,i}:=\left\{\begin{array}{c}
 \text { there exists a box } B'\in \mathrm{END}_{j-r-1}\text{ such that } B'\subseteq V_{3K^{j-1}}(z_i) \\
  \text{ and }\big|H_i(x)-H_i(z_{B'})\big|\ge\varepsilon_{r+1} \text{ for some } x\in V_{4K^{j-r-1}}(z_{B'})
 \end{array}\right\},
 \end{equation}
 \[
  \mathcal{E}_{r+1,\mathcal S}:=\bigcup_{i\in [d]}\mathcal{E}_{r+1,i}, \ \text{ and } \
 \mathcal{E}_{r+1}:= \bigcup_{d\ge K}\bigcup_{\mathcal S\in\mathrm{END}_{j-1, d}(B)}\mathcal{E}_{r+1,\mathcal S} \ \text{ for all } r\ge 1.
 \]

 For $P\in \mathcal{P}_{j, \mathcal{S}}(B), i\in [d], u\in \mathcal{T}_{P^{(i)},r}$, we have $B_u\in \mathrm{END}_{j-r-1}, B_u\subseteq V_{3K^{j-1}}(z_i)$ and $u\subseteq V_{4K^{j-r-1}}(z_{u})$ with $z_u:=z_{B_u}$. Noting that $\eta^{V}(x)=\left(\eta^{V_i}(x)+H_i(z_u)\right)+\left( H_i(x)-H_i(z_u) \right)$ for all $x\in u$, on the event $\mathcal{E}_{r+1,i}^c$, by the triangle inequality, $u$ is $(V,\lambda,\bar\alpha_u)$-open implies that it is $(V_i, \lambda+\varepsilon_{r+1}, \bar\alpha_u+H_i(z_u))$-open.
 Thus, by an analogous reasoning of \eqref{eq:1S-delta}, for all $\delta>0$ and $\mathcal S\in\mathrm{END}_{j-1, d}(B)$,
 \begin{equation}\label{eq:r+1S-delta}
 \mathbb P\left(\{\zeta_{r+1,\mathcal{S}} >\delta\}\cap\mathcal{E}_{r+1,\mathcal S}^c\right)\le \mathbb P\left(\frac1d \sum_{i=1}^{d}\xi'_{i,r,\lambda+\varepsilon_{r+1},j-1}>\delta\right),
 \end{equation}
 where
 $
 \zeta_{r+1,\mathcal{S}} :=\max \big\{Y_{P, r+1,\lambda, \bar\alpha}: P \in \mathcal{P}_{j, \mathcal{S}}(B)\big\},
 $ and
 \begin{equation}\label{eq:xi'_i}
 \xi'_{i,r,\lambda+\varepsilon_{r+1},j-1}:=\max_{P' \in \mathcal{P}_{j-1}(B_i)}\sum_{u\in \mathcal T_{P',r}}\theta_{P'}(u)1_{\big\{ u \text{ is tame and } (V_i,\lambda+\varepsilon_{r+1},\bar\alpha_u+H_i(z_{u}))\text{-open}\big\}}.
 \end{equation}

In addition, the following lemma is analogous to Lemma \ref{lem:E-1}.
\begin{lem}\label{lem:E-r}
	Let $K\ge C_5$. Then
	$
	\mathbb P\left(\mathcal{E}_{r+1} \right)
	\le e^{-c_r},
	$
	where $c_r=(\beta K)^r$ is set in \eqref{eq:beta-c_r}.
\end{lem}

\begin{proof}
	There are at most $K^7$ boxes in $\mathrm{END}_{j-1}$ intersecting with some paths in $\mathcal{P}_j(B)$. Denote them by $B_t$'s. For each $B_t$, denote $z_t:=z_{B_t}, V_t:=V_{4K^{j-1}}(z_t)$, and by $H_t$ the conditional expectation of $\eta^V$ given $\eta^V\big|_{ V_t^c\cup\partial V_t }$. Let $\mathcal{E}_{r+1,t}$ be the event as in \eqref{eq:Ei} with $t$ in place of $i$. 

	Note that $\mathcal{E}_{r+1}$ implies $\mathcal{E}_{r+1,t}$ for some $t$. It suffices to estimate the probability of $\mathcal{E}_{r+1,t}$ for all $t$.
	Setting $\ell=4K^{j-r-1}$ and $L=4K^{j-1}$ in Lemma \ref{lem:fluct-H}, for $K\ge C_5$, we have
	$
	C_3\sqrt{\frac{\ell}{L}}=C_3 K^{-\frac r2}\le \varepsilon_{r+1},
	$
	and for any box $B'$ in $\mathrm{END}_{j-r-1}$ and $B'\subseteq V_{3K^{j-1}}(z_t)$, we have $V_{4K^{j-r-1}}(z_{B'})\subseteq V_t^{\chi}$, therefore
	\[
	\mathbb P\Big( \big|H_t(x)-H_t(z_t)\big|\ge\varepsilon_{r+1} \text{ for some }  x\in V_{4K^{j-r-1}}(z_{B'}) \Big)\le
	4\exp\left\{-\frac{\varepsilon_{r+1}^2}{8C_3}K^r \right\}\le 4e^{-2c_r}.
	\]
	 Note that there are at most $\left( \frac{3K^{j-1}}{K^{j-r-3}} \right)^2\le 9K^{2r+4}$ boxes $B'$'s in $\mathrm{END}_{j-r-1}$ such that $B'\subseteq V_{3K^{j-1}}(z_t)$. By a union bound,
	$
	\mathbb P\left(\mathcal{E}_{r+1} \right)
	\le K^{7}\cdot 9K^{2r+4} \cdot 4e^{-2c_r}
	\le e^{-c_r},
	$
	completing the proof.
\end{proof}

\begin{proof}[Proof of Theorem \ref{thm:xi-bound}]
		We will apply induction on $r$, similar to the proof of \cite[Lemma 4.4]{MR3947326}. To this end, we will prove that the following hold for all $r\in[2,m-2]\cap\mathbb Z$ and $j\in [r+1,m-1]\cap\mathbb Z$.
	
		(i) Suppose $K\ge K_r(\lambda)$, $B\in \mathrm{END}_j$, $V_{4K^j}(z_{B})\subseteq V\subseteq V_{2N}$, and $\bar\alpha\in\mathscr A$. Then, 
		\[
		\mathbb{P}(\xi_{r,\lambda,\bar\alpha,j,B}>\delta)
		\le 2 e^{-c_{r-1}(\delta-\delta_r)}
		\quad \text{for all } \delta\ge\delta_r.
		\]
		
		(ii) Suppose $K\ge K_{r+1}(\lambda)$, $B\in \mathrm{END}_{j+1}$, $d\ge K$, $\left\{B_{i}\right\}_{i \in[d]} \in \mathrm{END}_{j, d}(B)$, and denote $\xi'_i:=\xi'_{i,r,\lambda+\varepsilon_{r+1},j}, i\in [d]$, defined in \eqref{eq:xi'_i}. Then,
		\[
		P\left(\frac1d \sum_{i=1}^{d}\xi'_{i}>\delta\right)
		\le \Big(K^{-9} e^{-\beta c_{r-1}(\delta-\delta_{r+1})}\Big)^d \text{ for all } \delta\ge\delta_{r+1}.
		\]
		
	In Step 1, we will show that (i) implies (ii). In Step 2, we will show (i) for $r+1$ and all $j \in[r+2, m-1] \cap \mathbb{Z}$, provided that (ii) holds for all \(j \in[r+1, m-1] \cap \mathbb{Z}\). In Step 3, we will show (i) holds for $r=2$ and $j\in [3,m-1]\cap\mathbb Z$.
	
	\textbf{Step 1.} Suppose (i) holds. We will prove (ii). 
	We can classify $\{B_{i}\}_{i\in [d]}$ into $432(\le 2^9=\beta^{-1})$ groups $\mathcal G_t$'s such that $V_i$'s in each group are disjoint, where $V_i=V_{4K^{j-1}}(z_i)$. Let $V_t$ be the union of $V_i$'s with $i$ such that $B_i\in\mathcal G_t$. Define
	$
	\mathcal F_{t}:=\big\{ \eta^{V} (x): x\in (V   \backslash V_{t})\cup\partial V_{t} \big\}.
	$
	Conditioned on $\mathcal F_{t}$, $\xi'_i$'s in each group $\mathcal G_t$ are mutually independent.
	Next, we will estimate $\mathbb{E}\left(e^{a(\xi'_i-\delta)} \big| \mathcal F_{t}\right)$, where $\delta\ge\delta_{r+1}$ and $a>0$. For each $i\in [d]$, we apply (i) to $\lambda+\varepsilon_{r+1}$, and have for $K\ge K_{r+1}(\lambda):=K
	_r(\lambda+\varepsilon_{r+1})$,
	\begin{equation}\label{eq:xi'i}
	\mathbb P\left(\xi'_i>\delta\big| \mathcal F_{t}\right)\le 2e^{-c_{r-1}(\delta-\delta_r)} \quad \text{ for all } \delta\ge\delta_r.
	\end{equation}
	Note that $0\le\xi'_i\le1$. It follows that for each $a>0$,
\begin{small}
	\begin{equation*}
	\mathbb E \left(e^{a\xi'_i}\Big| \mathcal F_{t}\right)
	\le  e^{a\delta_r}+\int_{\delta_r}^{1} \mathbb P\left(\xi'_i>z\big| \mathcal F_{t}\right)ae^{az} dz\\
	=e^{a\delta_r}\left( 1+2a\int_{0}^{1-\delta_r}e^{(a-c_{r-1})z}dz \right).
	\end{equation*}
\end{small}
Take $a=c_{r-1}$, then $\mathbb E \left(e^{a\xi'_i}\big| \mathcal F_{t}\right)\le (1+2c_{r-1})e^{c_{r-1}\delta_r}$. Hence for all $\delta\ge\delta_r$,
	\[
	\mathbb E \left(e^{a\left(\xi'_i-\delta\right)}\Big| \mathcal F_{t}\right)\le (1+2c_{r-1})e^{-c_{r-1}(\delta-\delta_r)}.
	\]
	Using the same argument from \eqref{eq:mathod-1} to \eqref{eq:method-2} as in Lemma \ref{lem:average-level0}, we obtain
	\begin{equation}\label{eq:4.3-1}
	P\left(\frac1d \sum_{i=1}^{d}\xi'_i>\delta\right)
	\le \Big( (1+2c_{r-1})e^{-c_{r-1}(\delta-\delta_r)} \Big)^{\beta d}.
	\end{equation}
	Recall $\Delta_r=\delta_{r+1}-\delta_r$ from \eqref{eq:Delta}. We get
	$
	\Big( (1+2c_{r-1})e^{-c_{r-1}(\delta_{r+1}-\delta_r)} \Big)^{\beta}\le K^{-9}.
	$
    Combined with \eqref{eq:4.3-1}, this implies (ii).
	
	\textbf{Step 2.} Assuming that (ii) holds for all \(j \in[r+1, m-1] \cap \mathbb{Z}\), we will show (i) for $r+1$ and all $j \in[r+2, m-1] \cap \mathbb{Z}$.  Similar to \eqref{eq:xi1-delta}, we have
	\begin{equation}\label{eq:xi-r+1-delta}
	\mathbb P\big(\xi_{r+1}>\delta\big)
	\le \sum_{d=K}^\infty\sum_{\mathcal S\in\mathrm{END}_{j-1,d}(B)} \mathbb P\big(\{\zeta_{r+1,\mathcal S}>\delta \}\cap\mathcal{E}_{r+1,\mathcal S}^c\big)+\mathbb P\big(\mathcal{E}_{r+1}\big).
	\end{equation}
	Note that $r+2\le j\le m-1$ implies $r+1\le j-1\le m-1$, then for $K\ge K_{r+1}(\lambda)$, we apply (ii) to $j-1$, and have
	\[
	\mathbb P\left(\frac1d \sum_{i=1}^{d}\xi'_{i,r,\lambda+\varepsilon_{r+1},j-1}>\delta\right)\le
	\Big(K^{-9} e^{-\beta c_{r-1}(\delta-\delta_{r+1})}\Big)^d \ \text{ for all } \delta\ge\delta_{r+1}.
	\]
	Combined with \eqref{eq:r+1S-delta}, this gives that for each $\mathcal S\in\mathrm{END}_{j-1,d}(B)$,
	\begin{equation}
	\mathbb P\big(\{\zeta_{r+1,\mathcal S}>\delta \}\cap\mathcal{E}_{r+1,\mathcal S}^c\big)
	\le \Big(K^{-9} e^{-\beta c_{r-1}(\delta-\delta_{r+1})}\Big)^d \ \text{ for all } \delta\ge\delta_{r+1}.
	\end{equation}
	Note that there are at most $K^{7d}$ sequences in $\mathrm{END}_{j-1,d}(B)$. By a union bound, the first term on the right hand side of \eqref{eq:xi-r+1-delta} is less than
	\begin{equation}\label{eq:3/2}
     \sum_{d=K}^{\infty}\Big(K^{-2}e^{-\beta c_{r-1}(\delta-\delta_{r+1})}\Big)^d
     \le \frac32e^{-c_{r}(\delta-\delta_{r+1})},
    \end{equation}
     since $K^{-2}e^{-\beta c_{r-1}(\delta-\delta_{r+1})}\le K^{-2}\le\frac13$ for $\delta\ge \delta_{r+1}$.
	Moreover, note that $\delta-\delta_{r+1}\le\frac12$ since $\delta\le1$ and $\delta_{r+1}\ge\frac12$, then by Lemma \ref{lem:E-r},
	\begin{equation}\label{eq:1/2}
	P\big(\mathcal{E}_{r+1}\big)\le e^{-c_r}\le \frac12e^{-c_{r}/2}\le \frac12e^{-c_{r}(\delta-\delta_{r+1})}.
	\end{equation}
	Plugging \eqref{eq:3/2} and \eqref{eq:1/2} into \eqref{eq:xi-r+1-delta}, we obtain
	$
	\mathbb P(\xi_{r+1}>\delta) \le 2e^{-c_{r}(\delta-\delta_{r+1})}.
	$
	That is, (i) holds for $r+1$ and all $j \in[r+2, m-1] \cap \mathbb{Z}$.
	
	\textbf{Step 3.} We will show that (i) holds for $r=2$ and $j\in [3,m-1]\cap\mathbb Z$. This follows lines in Step 1. We write $\xi'_i=\xi'_{i,1,\lambda+\varepsilon_2,j-1}, i\in [d]$ for brevity. Note that $j-1\in [2,m-1]\cap\mathbb Z$. Then applying Theorem \ref{thm:r=1} to $\lambda+\varepsilon_2$ and $K\ge K_2(\lambda)=K_1(\lambda+\varepsilon_2)$, we obtain
	\[
	\mathbb P\left(\xi'_i>\delta\big| \mathcal F_{t}\right) \le \exp\{ -K^{1/8} \} \ \text{ for all } \delta\ge\delta_{1},
	\]
	playing the role of \eqref{eq:xi'i}.
	Consequently, for $a=K^{1/8}$,
    \begin{align*}
    & \mathbb E \left(e^{a\xi'_i}\Big| \mathcal F_{t}\right)
    =\int_{0}^{1} \mathbb P\left(\xi'_i>z\big| \mathcal F_{t}\right)ae^{az} dz\\
    \le & e^{a\delta_1}-1+\int_{\delta_r}^{1} e^{-K^{1/8}}ae^{az} dz \le e^{a\delta_1}-1+e^{a-K^{1/8}}=e^{a\delta_1}.
    \end{align*}
     Hence, it holds that
    \[
     \mathbb E \left(e^{a\left(\xi'_i-\delta\right)}\Big| \mathcal F_{t}\right)
     \le e^{-K^{1/8}(\delta-\delta_1)}.
    \]
    Then, we have
    \[
    \mathbb P\left(\frac1d \sum_{i=1}^{d}\xi'_i>\delta\right)
    \le \left( e^{-\beta K^{1/8}(\delta-\delta_1)}\right)^d,
    \]
    as the counterpart of \eqref{eq:4.3-1}.
	Recall that $\delta_2-\delta_1=\Delta_1=\frac{9\log K}{\beta K^{1/8}}$, thus for $\delta\ge\delta_2$,
	\[
	\mathbb P\left(\frac1d \sum_{i=1}^{d}\xi'_i>\delta\right)
	\le \left( K^{-9}e^{-\beta K^{1/8}(\delta-\delta_2)}\right)^d.
	\]
    Consequently,
    \[
    \mathbb{P}(\xi_{2}>\delta)
    \le \sum_{d=K}^{\infty}\Big(K^{-2}e^{-\beta K^{1/8}(\delta-\delta_{2})}\Big)^d+e^{-c_1},
    \]
    as the counterpart of \eqref{eq:xi-r+1-delta}.
    With estimates similar to \eqref{eq:3/2} and \eqref{eq:1/2} for $r=2$, we conclude that for $K\ge K_2(\lambda)$,
	\[
	\mathbb{P}(\xi_{2}>\delta)
	\le 2 e^{-c_{1}(\delta-\delta_2)}
	\] 
	for all $\delta\ge\delta_2,$
	completing the proof.
\end{proof}

\subsection{Proof of Theorem \ref{thm:1.1} }

Recall that $v\in V_{2N}$ is $\lambda$-open if $\left| \eta^{V_{2N}}(v) \right|\le\lambda$, i.e., $(V_{2N},\lambda,0)$-open. Define
$
\tilde Y_{P,r,\lambda}:=\sum_{u\in \mathcal T_{P,r}}\theta_P(u)1_{\{u \text{ is tame and } \lambda\text{-open}\}},
$ and
\[
\tilde \xi_{r,\lambda,j,B}:=\max\big\{\tilde Y_{P,r,\lambda}: P\in \mathcal P_j(B)\big\}\ \text{ for all }  B\in\mathrm{END}_j.
\]
For $j\in [3,m-1]\cap\mathbb Z$ and $2\le r\le j-1$, let $K\ge K_{r}(\lambda)$. Applying Theorem \ref{thm:xi-bound} to $V=V_{2N}$, $\bar\alpha\equiv 0$ and $B\in\mathrm{END}_j$, we get
	\begin{equation}\label{eq:tilde-xi-bound}
	\mathbb{P}\left(\tilde\xi_{r,\lambda,j,B}>\delta\right)
	\le 2 e^{-c_{r-1}(\delta-\delta_r)} \ \text{ for all } \delta\ge\delta_{r+1},
	\end{equation}
recalling \eqref{eq:beta-c_r}, \eqref{eq:delta} for the definition of $c_r, \delta_r$.
Recall $\mathcal{P}^{\kappa, \delta, K}$ from \eqref{eq:kap-delta-K}.
For each $P\in\mathcal{P}^{\kappa, \delta, K}$, let $\big\{P^{(i)}:i\in[d_0]\big\}$ be the child-paths of $P$ in $\mathcal{SL}_{m-1}$ from Proposition \ref{prop:tree}. Recall that $L(u)$ is the depth of $u$ with $L(\rho)=0$. For a sub-path $u$ of $P^{(i)}$ in $\mathcal T_P$, denote by $L_i(u):=L(u)-1$ the level of $u$ in $\mathcal T_{P^{(i)}}$. By Lemma \ref{lem:untame-flow},
\begin{equation*}
\sum_{i=1}^{d_0}\sum_{u: 0 \leq L_i(u) \leq m-2}\frac{1}{d_0} \theta_{P^{(i)}}(u) 1_{\{u \text{ is untamed}\}}= \sum_{u: 1 \leq L(u) \leq m-1} \theta_{P}(u) 1_{\{u \text{ is untamed}\}}\leq 2 \delta m.
\end{equation*}
This implies that there is at least one child-path $P^{(i_0)}$ such that
\begin{equation}\label{eq:thm1-untamed}
\sum_{u: 0 \leq L_{i_0}(u) \leq m-2} \theta_{P^{(i_0)}}(u) 1_{\{u \text{ is untamed}\}}\leq 2 \delta m.
\end{equation}
Thus if there is a $\lambda$-open path $P$ in $\mathcal{P}^{\kappa, \delta, K}$, there would exist a $\lambda$-open path $\tilde P$ in $\mathcal P_{m-1}(B)$ for some $B\in\mathrm{END}_{m-1}$ such that \eqref{eq:thm1-untamed} holds with $P^{(i_0)}$ replaced with $\tilde P$ and $L_{i_0}(u)$ replaced with $\tilde L(u)$, the depth of $u$ in $\mathcal T_{\tilde P}$. Note that if $\tilde P$ is $\lambda$-open, then all the sub-paths are $\lambda$-open, which leads to
\begin{align*}
m-1
&=\sum_{u: 0 \leq\tilde L(u) \leq m-2} \theta_{\tilde P}(u) 1_{\{u \text{ is } \lambda\text{-open}\}}\\
&=\sum_{r=0}^{m-2} \tilde Y_{\tilde P,r,\lambda}+\sum_{u: 0 \leq \tilde L(u) \leq m-2} \theta_{\tilde P}(u) 1_{\{u \text{ is untamed}\}}
\le \sum_{r=0}^{m-2}\tilde \xi_{r,\lambda,m-1,B}+2\delta m.
\end{align*}
By the above inequality, in order to prove that for some $\delta>0, K>0$,
\begin{equation}\label{eq:P-kap-delta-K}
\lim _{N \rightarrow \infty}\mathbb{P}\big( P \text { is } \lambda\text{-open for some } P \in \mathcal{P}^{\kappa, \delta, K}  \big)=0,
\end{equation}
it is sufficient to show that there exists $\delta>0, K(\lambda,\delta)\in (0,\infty)$ such that for $K\ge K(\lambda,\delta)$,
\begin{equation}\label{eq:sum-r}
\lim_{m\rightarrow\infty}\mathbb P\left( \sum_{r=0}^{m-2}\tilde \xi_{r,\lambda,m-1,B}\ge m-1-2\delta m \text{ for some } B\in\mathrm{END}_{m-1} \right)=0.
\end{equation}

Recall that $\{\varepsilon_r: r\ge 0\}$ is set in \eqref{eq:epsilon} and $K_r(\lambda)=K_0\left(\lambda+\sum_{i=1}^{r}\varepsilon_i\right)$ is defined in \eqref{eq:Kr}. Noting that $\sum_{i=1}^{\infty}\varepsilon_i<\infty$ and $K_0(\cdot)$ is a increasing function, one has $K_{\infty}(\lambda):=K_0\left(\lambda+\sum_{i=1}^{\infty}\varepsilon_i\right)<\infty$.
Recall \eqref{eq:Delta} for the definition of $\Delta_r$. There exists $K(\lambda,\delta)\ge K_{\infty}(\lambda)$ such that the following inequality holds for all $K\ge K(\lambda,\delta)$,
\[
\sum_{r=1}^{\infty}\Delta_r=\frac{9\log K}{\beta K^{1/8}}+\sum_{r=1}^{\infty}\frac{\log(1+2c_{r})+9\beta^{-1}\log K}{c_{r}}\le \delta.
\]
Consequently, for $K\ge K(\lambda,\delta)$, we have $\delta_r\le \frac12+\delta$ for all $r\ge 0$. As $m\ge \delta^{-1}$,
\begin{equation}\label{eq:sum-r-}
\sum_{r=0}^{m-2}\left(\delta_r+\frac{1}{2^{r+2} } \left(1-8\delta\right)m\right)
< m-1-2\delta m,
\end{equation}
where we set $\delta\in(0,\frac18)$. Since $\kappa N<K^{m+2}$, there are at most $(K^6/\kappa)^2$ boxes in $\mathrm{END}_{m-1}$.
By a union bound and \eqref{eq:sum-r-}, for $K\ge K(\lambda,\delta)$ and $m\ge \delta^{-1}$,
\begin{equation}\label{eq:sum-r-2}
\mathbb P\left( \sum_{r=0}^{m-2}\tilde \xi_{r,\lambda,m-1,B}\ge m-1-2\delta m \text{ for some } B\in\mathrm{END}_{m-1} \right)
\le \frac{K^{12}}{\kappa^2}\sum_{r=0}^{m-2}p_{m,r},
\end{equation}
where
$
p_{m,r}=\mathbb P\left(\tilde \xi_{r,\lambda,m-1,B}>\delta_r+\frac{1}{2^{r+2} } \left(1-8\delta\right)m\right).
$
As $m\ge \frac{8}{1-8\delta}$, for $r=0,1$, we have $\frac{1}{2^{r+2} } \left(1-8\delta\right)m\ge 1$, implying $p_{m,r}=0$. Applying \eqref{eq:tilde-xi-bound} to $j=m-1$, then for all $2\le r\le m-2$,
\[
p_{m,r}\le 2\exp\left\{ -\frac18\left(1-8\delta\right)\left(\frac{\beta K}{2}\right)^{r-1}m \right\}.
\]
Furthermore, $\beta K/2\ge 2^{22}$ implies that $\frac18\left(\frac{\beta K}{2}\right)^{r-1}\ge r$ for all $r\ge 2$. Thus,
\begin{equation}\label{eq:decay-rate}
\sum_{r=0}^{m-2}p_{m,r}\le 2\sum_{r=2}^{m-2}e^{ -\left(1-8\delta\right)mr}\le \frac{2}{e^{\left(1-8\delta\right)m}-1},
\end{equation}
which converges to $0$ as $m\rightarrow\infty$. Combined with \eqref{eq:sum-r-2}, this implies \eqref{eq:sum-r}.

Especially, set $\delta=\frac{1}{16}$.
Then for $\lambda\ge\lambda_0$, there exists $K(\lambda)=e^{b\lambda^2}\ge K(\lambda,1/16)$ for some $b=b(\lambda_0)>0$ such that for $\delta=\frac{1}{16}$ and $K\ge K(\lambda)$, \eqref{eq:P-kap-delta-K} holds. Let 
\begin{equation}\label{eq:eps-lambd}
\epsilon(\lambda)=\frac{1}{16K(\lambda)^2k(\lambda)},
\end{equation}
then $\mathcal{P}_{N}^{\kappa,\epsilon(\lambda)}=\mathcal{P}^{\kappa, 1/16, K(\lambda)}$ and \eqref{eq:P-kap-delta-K} implies \eqref{eq:complement-event}. We conclude the proof of Theorem \ref{thm:1.1}.

\bigskip
\noindent
{\bf Acknowledgments:}
This work is supported by NSF of China 11771027. We would like to thank Jian Ding for his suggestions and helpful discussions.
\bibliographystyle{abbrv}
\bibliography{GFFref}

\end{document}